\newif\ifdraft

\documentclass[12pt]{amsart}

\usepackage{enumitem}
\usepackage{verbatim}
\usepackage{xcolor}
\usepackage[margin=1.3in]{geometry}
\usepackage[normalem]{ulem}
\usepackage{caption}
\usepackage{wrapfig,graphicx}
\usepackage{hyperref}

\ifdraft
\usepackage{fancyhdr,datetime2}
\fancypagestyle{firstpage}{%
	
	\fancyhf{}
	\fancyhead[C]{\color{red}\fbox{\textbf{DRAFT} (\DTMnow) -- NOT FOR DISTRIBUTION}}
	\fancyfoot[C]{\thepage}
}
\fi

\newcommand{\deleted}[1]
{}
\newcommand{\modified}
{\color{blue}}

\newcommand{\ph}{\varphi}
\newcommand{\eps}{\epsilon}
\newcommand{\ulim}{\varlimsup}

\newcommand{\EE}{\mathbb{E}}
\newcommand{\NN}{\mathbb{N}}
\newcommand{\ZZ}{\mathbb{Z}}
\newcommand{\RR}{\mathbb{R}}

\newcommand{\BBB}{\mathcal{B}}

\newcommand{\CCC}{\mathcal{C}}
\newcommand{\DDD}{\mathcal{D}}

\newcommand{\LLL}{\mathcal{L}}
\newcommand{\FFF}{\mathcal{F}}
\newcommand{\EEE}{\mathcal{E}}

\newcommand{\SSS}{\mathcal{S}}

\newcommand{\RRR}{\mathcal{R}}
\newcommand{\PPP}{\mathcal{P}}
\newcommand{\GGG}{\mathcal{G}}

\newcommand{\one}{\mathbf{1}}

\newcommand{\Ws}{W^\mathrm{s}}
\newcommand{\Wu}{W^\mathrm{u}}
\newcommand{\Wsl}{\Ws_\mathrm{loc}}
\newcommand{\Wul}{\Wu_\mathrm{loc}}

\newcommand{\ms}{m^\mathrm{s}}
\newcommand{\mmu}{m^\mathrm{u}}

\newcommand{\xminus}{x.}

\newcommand{\uQ}{\overline{Q}}
\newcommand{\lQ}{\underline{Q}}

\newtheorem{theorem}{Theorem}[section]
\newtheorem{lemma}[theorem]{Lemma}
\newtheorem{proposition}[theorem]{Proposition}
\newtheorem{corollary}[theorem]{Corollary}

\newtheorem{thma}{Theorem}
\newtheorem*{thma*}{Theorem}

\theoremstyle{remark}
\newtheorem{remark}[theorem]{Remark}

\theoremstyle{definition}
\newtheorem{definition}[theorem]{Definition}

\numberwithin{equation}{section}

\begin{document}
	
	\title[Equilibrium measures for shift spaces]{Equilibrium measures for two-sided shift spaces via dimension theory}
	\author{Vaughn Climenhaga}
	\address{Dept.\ of Mathematics, University of Houston, Houston, TX 77204}
	\email{climenha@math.uh.edu}
	\author{Jason Day}
	\address{Dept.\ of Mathematics, University of Houston, Houston, TX 77204}
	\email{jjday@uh.edu}
	\date{\today}
	\thanks{The authors were partially supported by NSF grants DMS-1554794 and DMS-2154378. VC was partially supported by a Simons Foundation Fellowship.}

\begin{abstract}
Given a two-sided shift space on a finite alphabet and a continuous potential function, we give conditions under which an equilibrium measure can be described using a construction analogous to Hausdorff measure that goes back to the work of Bowen. This construction was previously applied to smooth uniformly and partially hyperbolic systems by the first author, Pesin, and Zelerowicz. 
Our results here apply to all subshifts of finite type and H\"older continuous potentials, but extend beyond this setting, and we also apply them to shift spaces with synchronizing words.
\end{abstract}
	
\maketitle

\ifdraft
\thispagestyle{firstpage}
\pagestyle{fancy}
\fi

\section{Introduction and main results}

\subsection{Background}

In the study of hyperbolic dynamical systems, thermodynamic formalism allows us to identify certain \emph{equilibrium measures} whose ergodic and statistical properties provide insight into the underlying dynamics.   Given a compact metric space $X$, a continuous map $f\colon X\to X$, and a continuous \emph{potential} function $\ph\colon X\to \RR$, an $f$-invariant Borel probability measure on $X$ is an equilibrium measure if it maximizes $h_\mu(f) + \int \ph\,d\mu$, where $h_\mu(f)$ is the measure-theoretic entropy.

\begin{wrapfigure}{r}{.36\textwidth}
\captionsetup{justification=raggedright,margin=12pt}
\hfill
\includegraphics[width=.32\textwidth]{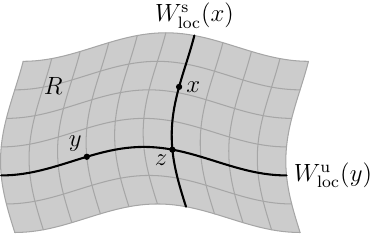}
\caption{A rectangle.}
\label{fig:rectangle}
\end{wrapfigure}

If $X$ is a locally maximal hyperbolic set for a diffeomorphism $f$, then given any two nearby points $x,y\in X$, their local stable and unstable manifolds intersect in a unique point $z\in \Wsl(x) \cap \Wul(y)$. A set $R\subset X$ is a \emph{rectangle} if for all $x,y\in R$, this point $z$ is defined and lies in $R$, as in Figure \ref{fig:rectangle} (in general $R$ may have empty interior). This provides a \emph{product structure} on $R$. This paper is concerned with the behavior of equilibrium measures with respect to this product structure.

With $(X,f)$ as in the previous paragraph, and $\ph \colon X\to \RR$ H\"older continuous, it is known that every transitive component of $X$ has a unique equilibrium measure $\mu$, and that there are \emph{leaf measures} $\ms_x, \mmu_x$ such that on every rectangle, $\mu$ is equivalent to $\ms_x\times \mmu_x$.

It is possible to describe these leaf measures as analogues of Hausdorff measure, where the ``refining'' in the definition of Hausdorff measure is done dynamically rather than geometrically. The idea of treating topological entropy and pressure as analogues of Hausdorff dimension goes back to Bowen \cite{rB73} and to Pesin and Pitskel' \cite{PP84}. The idea of constructing leaf measures of the measure of maximal entropy via a Hausdorff measure construction first appeared in \cite{uH89,bH89}; the general theory was developed by the first author, Pesin, and Zelerowicz \cite{CPZ,CPZ2,direct-srb}.
	
In this paper, we take some steps toward the setting of non-uniform hyperbolicity, or hyperbolicity with singularities, by considering two-sided shift spaces that are not necessarily subshifts of finite type. We give general conditions under which this dimension theoretic construction of leaf measures can still be used to construct an equilibrium measure with local product structure. Using these conditions, we obtain the following application (see \S\S\ref{sec:setting}--\ref{sec:classes} for precise definitions).

\begin{thma}\label{thm:spec}
Let $X$ be a two-sided shift space on a finite alphabet with the specification property, and let $\ph\colon X\to\RR$ be a continuous potential function with the Walters property. Then the unique equilibrium measure $\mu$ for $(X,\ph)$ has local product structure in the sense of \S\ref{sec:prod} below.
\end{thma}

Theorem \ref{thm:spec} follows from Theorem \ref{thm:sync} in \S\ref{sec:synchronizing-words} (see Remark \ref{rmk:spec}), which applies to the broader class of shift spaces with synchronizing words satisfying a certain summability condition \eqref{eqn:sync-assumption}. In a forthcoming paper \cite{CD-billiards}, we will explore applications of our general results to the measure of maximal entropy for dispersing billiards that was recently studied by Baladi and Demers \cite{BD20}.

\begin{remark}\label{rmk:known}
The conclusion of Theorem \ref{thm:spec} can be deduced from results in the literature when $X$ is a mixing subshift of finite type and $\ph$ is H\"older continuous \cite{Bow08,nH94,Lep} or at least has the (weaker) Walters property \cite{pW78}, and also when $X$ has specification and $\ph$ is H\"older continuous \cite{vC18}. The novelty here is the ability to simultaneously weaken SFT to specification, and H\"older to Walters, as well as to provide a framework for going beyond specification (see \S\ref{sec:synchronizing-words}). Uniqueness was already known in the setting of Theorem \ref{thm:spec} \cite{rB745}, but not local product structure. 

It is also worth mentioning \cite{BO23}, which establishes a certain local product structure property for equilibrium measures of non-uniformly hyperbolic diffeomorphisms with potentials satisfying a ``Grassmann--H\"older'' continuity condition.
\end{remark}

The main general results of this paper are the following.
\begin{itemize}
\item \S\ref{sec:build-leafs}, \eqref{eqn:shift-mxu} and \eqref{eqn:shift-mxs}: Construction of the leaf measures $\mmu_x$ and $\ms_x$.
\item \S\ref{sec:finite}, Theorem \ref{thm:finite}: 
Conditions for leaf measures to be positive and finite.
\item \S\ref{sec:dynamics}, Theorem \ref{thm:dynamics}: 
Scaling properties of leaf measures under $\sigma^{\pm 1}$.
\item \S\ref{sec:holonomies}, Theorem \ref{thm:shift-cts}: 
Scaling properties of leaf measures under holonomies (sliding between stable leaves along unstable leaves, and vice versa).\footnote{These require certain conditions on a family of rectangles, which are satisfied
if all rectangles in the family are open and the potential has the Walters property (Proposition \ref{prop:open-hol}), or for an arbitrary family of rectangles if the potential is constant (Proposition \ref{prop:zero-hol}).}
\item \S\ref{sec:product-measure}, Theorems \ref{thm:q-indep}, \ref{thm:return-map}, and \ref{thm:push-product}: 
Construction of a product measure $\lambda$, an invariant measure $\mu$ with local product structure, and conditions for $\mu$ to give an equilibrium measure.
\item \S\ref{sec:gibbs}, Theorem \ref{thm:gibbs}: 
Gibbs-type estimates for $\ms_x, \mmu_x, \lambda, \mu$.
\end{itemize}

\subsection*{Acknowledgments} We are grateful to the anonymous referee for a careful reading and for a number of helpful comments that improved the paper, especially in Theorem \ref{thm:sync} and its proof.

\subsection{Description of setting}\label{sec:setting}

Let $A$ be a finite set (the \emph{alphabet}), and equip the set $A^\ZZ$ of all bi-infinite sequences over $A$ with the metric $d(x,y) = 2^{-\min \{|n| : x_n \neq y_n\}}$. The \emph{shift map} $\sigma \colon A^\ZZ\to A^\ZZ$ is defined by $(\sigma x)_n = x_{n+1}$. A \emph{(two-sided) shift space} is a closed $\sigma$-invariant set $X\subset A^\ZZ$. Given such a shift space $X$, we will study the thermodynamic formalism of the homeomorphism $\sigma \colon X\to X$ equipped with a continuous \emph{potential function} $\ph\colon X\to \RR$. To describe this, we first need some more terminology and notation.

A \emph{word} over $A$ is a finite string of symbols $w\in A^n$ for some $n$; we write $|w| = n$ for the \emph{length} of $w$. (The case $n=0$ gives the \emph{empty word}.)
Given $x\in A^\ZZ$ and $i,j\in \ZZ$, we write $x_{[i,j)}$ for the word $x_i x_{i+1} \cdots x_{j-1}$, and similarly for $(i,j]$, $(i,j)$, and $[i,j]$.
Given a shift space $X$ and a word $w\in A^n$, the \emph{(forward) cylinder} of $w$ is
\[
[w]^+ := \{x\in X : x_{[0,n)} = w\},
\]
and the \emph{language} of $X$ is
\[
\LLL := \bigcup_{n=0}^\infty \LLL_n,
\quad\text{where } \LLL_n := \{w\in A^n : [w]^+ \neq \emptyset \}.
\]
Given a potential function $\ph\colon X\to \RR$, we assign a weight to each word $w\in \LLL$ by
\begin{equation}\label{eqn:Phi}
\Phi(w) := \sup_{x\in [w]^+} S_n\ph(x),
\text{ where } S_n\ph(x) := \sum_{k=0}^{n-1} \ph(\sigma^k x).
\end{equation}
The \emph{$n$th partition sum} associated to $(X,\sigma,\ph)$ is
\begin{equation}\label{eqn:Lambda}
\Lambda_n := \sum_{w\in \LLL_n} e^{\Phi(w)}.
\end{equation}
Given $m,n\in \NN$ one quickly sees that
\begin{equation}\label{eqn:sub-mult}
\Lambda_{m+n} = \sum_{u\in \LLL_m} \sum_{\substack{v\in \LLL_n \\ uv\in \LLL_{m+n}}} e^{\Phi(uv)}
\leq \sum_{u\in \LLL_m} \sum_{\substack{v\in \LLL_n \\ uv\in \LLL_{m+n}}} e^{\Phi(u) + \Phi(v)}
\leq \Lambda_m \Lambda_n,
\end{equation}
and so Fekete's lemma guarantees that the following limit exists:
\begin{equation}\label{eqn:P}
P := \lim_{n\to\infty} \frac 1n \log \Lambda_n.
\end{equation}
The number $P = P(X,\sigma,\ph)$ is the \emph{topological pressure}. 
Writing $M_\sigma(X)$ for the space of $\sigma$-invariant Borel probability measures on $X$, and $h\colon M_\sigma(X) \to [0,\infty)$ for the measure-theoretic entropy function, the variational principle \cite[Theorem 9.10]{pW82} states that
\[
P = \sup_{\mu \in M_\sigma(X)} \Big( h(\mu) + \int \ph\,d\mu\Big).
\]
A measure achieving this supremum is an \emph{equilibrium measure}. There is always at least one equilibrium measure because the entropy function is upper semi-continuous. 

\subsection{Local product structure}\label{sec:prod}

We will be concerned with the structure and properties of equilibrium measures, and in particular with their product structure in terms of stable and unstable sets.

Given a shift space $X$, we say that $R\subset X$ has \emph{product structure} (or is a \emph{rectangle}) if for every $x,y\in R$, the \emph{bracket} $[x,y]$ given by
\begin{equation}\label{eqn:lps}
	[x,y]_n := \begin{cases} x_n & n \leq 0, \\ y_n & n\geq 0 \end{cases}
\end{equation}
is defined (which requires $x_0 = y_0$) and lies in $R$. In this case, we fix $q\in R$ and consider the sets
\begin{equation}\label{eqn:WRq}
\begin{aligned}
\Wu_R(q) &:= \{x\in R : x_n = q_n \text{ for all } n\leq 0 \}, \\
\Ws_R(q) &:= \{x\in R : x_n = q_n \text{ for all } n\geq 0 \},
\end{aligned}
\end{equation}
and then consider the bijection
\begin{equation}\label{eqn:iota}
\begin{aligned}
\iota_{q} \colon \Wu_R(q) \times \Ws_{R}(q) &\to R, \\
(x,y) &\mapsto [y,x].
\end{aligned}
\end{equation}
We say that a measure $m$ on $R$ is a \emph{product measure} if $m = (\iota_q)_* (\mmu_q \times \ms_q)$ for some measures $\mmu_q$ on $\Wu_R(q)$ and $\ms_q$ on $\Ws_R(q)$.

\begin{definition}\label{def:lps}
A probability measure $\mu$ on $X$ has \emph{local product structure} if there exist a sequence of product measures $m_n$ on rectangles $R_n$ such that we have
\begin{itemize}
\item $\mu|_{R_n} \ll m_n$ for every $n\in \NN$; and
\item $\mu(\bigcup_{n\in\NN} R_n) = 1$.
\end{itemize}
\end{definition}

\subsection{Classes of shift spaces and potentials}\label{sec:classes}

To go beyond mere existence and deduce uniqueness, local product structure, or other stronger properties for equilibrium measures, one needs to know more about the shift space $X$ and the potential $\ph$. We gather some basic definitions and facts here. 
\begin{itemize}[itemsep=.5ex]
\item The \emph{$n$th variation} of $\ph$ is $V_n(\ph) := \sup \{ |\ph(x) - \ph(y)| : x_{(-n,n)} = y_{(-n,n)} \}$.
\item Every continuous $\ph\colon X\to\RR$ is uniformly continuous: $V_n(\ph)\to 0$ as $n\to\infty$.
\item $\ph$ is \emph{H\"older continuous} if $V_n(\ph)$ decays exponentially fast: there are $C,\alpha>0$ such that $V_n(\ph) \leq C e^{-\alpha n}$ for all $n\in\NN$.
\item $\ph$ is \emph{Dini continuous} (or has \emph{summable variations}) if $\sum_n V_n(\ph)<\infty$.
\item $\ph$ has the \emph{Bowen property} if there is $C>0$ such that
for every $n\in \NN$ and every $x,y\in X$ with $x_{[0,n)} = y_{[0,n)}$, we have $|S_n\ph(x) - S_n\ph(y)| \leq C$.
\item $\ph$ has the \emph{Walters property} if for every $\eps>0$, there is $k\in\NN$ such that for all $n\in \NN$ and $x,y\in X$ with $x_{[-k,n+k)} = y_{[-k,n+k)}$, we have $|S_n\ph(x) - S_n\ph(y)| \leq \eps$.
\end{itemize}
These properties have the following relationships:
\[
\text{H\"older}
\quad\Rightarrow\quad
\text{Dini}
\quad\Rightarrow\quad
\text{Walters}
\quad\Rightarrow\quad
\text{Bowen}.
\]
The Bowen property has the following immediate consequence:
\begin{equation}\label{eqn:almost-add}
\Phi(v) + \Phi(w) - C \leq \Phi(vw) \leq \Phi(v) + \Phi(w)
\text{ for all } v,w\in \LLL \text{ with } vw\in \LLL.
\end{equation}
A shift space $X\subset A^\ZZ$ is a \emph{subshift of finite type} (SFT) if
there is a finite set of \emph{forbidden words} $\FFF \subset A^* := \bigcup_{n=0}^\infty A^n$ such that
\begin{equation}\label{eqn:SFT}
X = \{ x\in A^\ZZ : x_{[i,j]} \notin \FFF \text{ for all } i,j\in \ZZ \}.
\end{equation}
Every topologically mixing SFT has the following \emph{specification property}: there is $t\in\NN$ such that for all $v,w\in \LLL$, there is $u\in \LLL_t$ such that $vuw\in \LLL$. As mentioned in Remark \ref{rmk:known}, the following facts were already known.
\begin{itemize}
\item If $X$ has specification and $\ph$ has the Bowen property, then there is a unique equilibrium measure \cite{rB745}, and it has the K property \cite{fL77,bC22}.
\item If $X$ is a mixing SFT and $\ph$ has the Walters property, then the unique equilibrium measure has local product structure and the Bernoulli property; see \cite{Bow08,nH94,Lep} for H\"older continuous $\ph$, and \cite{pW78} for the extension to the Walters property.
\item If $X$ has specification and $\ph$ is H\"older continuous, then the unique equilibrium measure has local product structure and the Bernoulli property \cite{vC18}.
\end{itemize}
Theorem \ref{thm:spec} extends the ``local product structure'' result in the last item to the case when $\ph$ has the Walters property. We do not address the question of whether the unique equilibrium measure is Bernoulli in this setting.

In the classical results above, an important role is played by the fact that the partition sums in \eqref{eqn:Lambda} admit \emph{uniform counting bounds} $e^{nP} \leq \Lambda_n \leq Q e^{nP}$, where $Q<\infty$ is independent of $n$, as well as by the existence of ``large'' sets with product structure. An important feature of our approach is that our general results rely only on these rather flexible assumptions, rather than more restrictive assumptions on the structure of $X$.

\subsection{Construction of the leaf measures}\label{sec:build-leafs}

The definition of topological pressure in \eqref{eqn:P} gave it as an exponential growth rate, analogous to box (capacity) dimension. In dimension theory one can also study Hausdorff dimension, which is a critical value rather than a growth rate; more precisely, one defines a one-parameter family of Hausdorff measures, whose total weight jumps from $\infty$ to $0$ at a particular value of the parameter. This value is the Hausdorff dimension; Pesin and Pitskel' \cite{PP84} gave an analogous definition of topological pressure, building on earlier work of Bowen \cite{rB73} for topological entropy.

In the symbolic setting, this definition goes as follows. Given $N\in \NN$, let $\LLL_{\geq N} = \{w\in \LLL : |w| \geq N\}$ denote the set of words of length at least $N$, and given $Z\subset X$, consider the following set of covers of $Z$ by cylinders corresponding to such words:
\begin{equation}\label{eqn:E+x}
\EE^+(Z,N) := \Big\{ \EEE\subset \LLL_{\geq N} : Z \subset \bigcup_{w\in \EEE} [w]^+ \Big\},
\end{equation}
Then to each $\alpha\in \RR$, associate an outer measure $m_\alpha$ on $X$ defined by
\begin{equation}\label{eqn:ma}
m_\alpha(Z) = \lim_{N\to\infty} \inf \Big\{ \sum_{w\in \EEE} e^{\Phi(w) - |w|\alpha} : \EEE \subset \EE^+(Z,N) \Big\}.
\end{equation}
The topological pressure $P\in \RR$ is the unique real number with the property that $m_\alpha(X) = \infty$ for all $\alpha < P$, and $m_\alpha(X) = 0$ for all $\alpha > P$. Now one may naturally ask whether $m_P$ has anything to do with the equilibrium measure(s) of $(X,\sigma,\ph)$. There are two problems that arise immediately:
\begin{enumerate}
\item a priori we could have $m_P(X) = 0$ or $m_P(X) = \infty$;
\item the outer measure $m_P$ does not restrict to a Borel measure unless $m_P \equiv 0$. Indeed, if we consider for each $x\in X$ the \emph{local unstable set} (or \emph{leaf})
\[
\Wul(x) = \{y\in X : y_k = x_k \text{ for all } k\leq 0\},
\] 
then $\EE^+(\Wul(x),N) = \EE^+([x_0]^+,N)$ for all $N$ if $X$ is a subshift of finite type, and thus $m_P(\Wul(x)) = m_P([x_0]^+)$, from which one can quickly deduce that $m_P$ is not additive on the Borel $\sigma$-algebra.
\end{enumerate}
The first problem is addressed in \S\ref{sec:finite}. For now we address the second problem, by defining a measure not on all of $X$, but on a local unstable set.
We modify the definition in \eqref{eqn:ma} slightly: instead of using weights $\Phi(w)$, we fix $x\in X$ and consider for each $w\in \LLL$ the quantity\footnote{If $\Wul(x) \cap [w]^+ = \emptyset$, then $\Phi_x^+(w) = -\infty$ and $e^{\Phi_x^+(w)} = 0$.}
\begin{equation}\label{eqn:P+x}
\Phi_x^+(w) = \sup \{ S_{|w|}\ph(y) : y\in \Wul(x) \cap [w]^+ \}.
\end{equation}
Then define the unstable leaf measure as
\begin{equation}\label{eqn:shift-mxu}
\mmu_x(Z) = \lim_{N\to\infty} \inf \Big\{ \sum_{w\in \EEE} e^{\Phi_x^+(w)-|w|P} : \EEE \in \EE^+(Z,N) \Big\}.
\end{equation}
Given $w\in \LLL_{\geq N}$, the set $\Wul(x) \cap [w]^+$ has diameter $\leq 2^{-N}$, which goes to $0$ as $N\to\infty$, and thus \cite[Lemma 2.14]{direct-srb} implies that $\mmu_x$ defines a Borel measure on $\Wul(x)$.

By reversing the direction of everything, we can analogously define a Borel measure $\ms_x$ on the \emph{local stable set} (or \emph{leaf})
\[
\Wsl(x) = \{y\in X : y_k = x_k \text{ for all } k\geq 0\}.
\]
For the definition of $\ms_x$ we work with the backwards cylinders corresponding to words $w\in A^n$:
\[
[w]^- = \{x\in X : x_{(-n,0]} = w\}.
\]
Given $Z \subset \Wsl(x)$ and $N\in \NN$, we consider the collection of covers
\begin{equation}\label{eqn:E-x}
\EE^-(Z,N) = \Big\{ \EEE\subset \LLL_{\geq N} : Z \subset \bigcup_{w\in \EEE} [w]^- \Big\}.
\end{equation}
Using the weight function
\begin{equation}\label{eqn:P-x}
\Phi_x^-(w) = \sup \Big\{ \sum_{k=0}^{|w|-1} \ph(\sigma^{-k} y) : y\in \Wsl(x) \cap [w]^- \Big\},
\end{equation}
we define the leaf measure $\ms_x$ by
\begin{equation}
	\label{eqn:shift-mxs}
	\ms_x(Z) = \lim_{N\to\infty} \inf \Big\{ \sum_{w\in \EEE} e^{\Phi_x^-(w) - |w|P} : \EEE \in \EE^-(Z,N) \Big\}.
\end{equation}
Once again, \cite[Lemma 2.14]{direct-srb} guarantees that $\ms_x$ is a Borel measure on $\Wsl(x)$. Now we must provide some way of guaranteeing that the leaf measures $\mmu_x$ and $\ms_x$ are positive and finite; we will do this in the next section. We conclude this section by observing that the set $\Wul(x)$ and the measure $\mmu_x$ only depend on $x_{(-\infty,0]}$; similarly, $\Wsl(x)$ and $\ms_x$ depend only on $x_{[0,\infty)}$. This is immediate from the definitions but deserves to be recorded:

\begin{proposition}\label{prop:xy}
If $x_{(-\infty,0]} = y_{(-\infty,0]}$ then $\Wul(x) = \Wul(y)$ and $\mmu_x = \mmu_y$. If $x_{[0,\infty)} = y_{[0,\infty)}$ then $\Wsl(x) = \Wsl(y)$ and $\ms_x = \ms_y$.
\end{proposition}

\begin{remark}\label{rmk:1-sided}
We will occasionally abuse notation slightly by using the notation $\Wul(x), \mmu_x$ when $x = \cdots x_{-2} x_{-1} x_0$ is a backward-infinite sequence, and $\Wsl(x), \ms_x$ when $x = x_0 x_1 x_2 \cdots$ is forward-infinite. By Proposition \ref{prop:xy}, the meaning of this notation is well-defined by extending such an $x$ to any bi-infinite sequence in $X$.
\end{remark}

\subsection{Conditions for finiteness}\label{sec:finite}

As the arguments in \cite{CPZ} reveal, the key property needed to guarantee positivity and finiteness of the leaf measures is uniform control of various version of the partition sums $\Lambda_n$ from \eqref{eqn:Lambda}. Start by observing that \eqref{eqn:P} means that $\Lambda_n \approx e^{nP}$ for large $n$, but leaves open the possibility that $\Lambda_n e^{-nP}$ can approach $0$ or $\infty$ subexponentially. In fact, by Fekete's lemma, the submultiplicativity property $\Lambda_{m+n} \leq \Lambda_m \Lambda_n$ guarantees that
\begin{equation}\label{eqn:Ln-geq}
\Lambda_n \geq e^{nP} \text{ for all } n;
\end{equation}
however, there is no universal upper bound. We define
\begin{equation}\label{eqn:Q}
\uQ := \varlimsup_{n\to\infty} \Lambda_n e^{-nP} \in [1,\infty].
\end{equation}
To obtain good bounds on $\mmu_x$, we will need $\uQ<\infty$. This is known to be true for transitive SFTs with H\"older continuous potentials, and in many other settings as well; see for example \S\ref{sec:spec} and \cite[Proposition 5.3]{CT13}. We will also need good control of partition sums restricted to $\Wul(x)$: let
\begin{equation}\label{eqn:Lambda+}
\Lambda_n^+(x) := \sum_{w\in \LLL_n} e^{\Phi_x^+(w)},
\end{equation}
recalling that words with $\Wul(x) \cap [w]^+ = \emptyset$ make no contribution to the sum, and then define
\begin{equation}\label{eqn:Q+}
\uQ^+_x := \varlimsup_{n\to\infty} \Lambda_n^+(x) e^{-nP},
\qquad
\lQ^+_x := \varliminf_{n\to\infty} \Lambda_n^+(x) e^{-nP}.
\end{equation}
Clearly we have $\lQ^+_x \leq \uQ^+_x \leq \uQ$, but we have no a priori lower bound on $\lQ^+_x$. Replacing $+$ with $-$ in \eqref{eqn:Lambda+} and \eqref{eqn:Q+} yields analogous quantities $\lQ^-_x \leq \uQ^-_x \leq \uQ$. The following result is proved in \S\ref{sec:shift-mxu-pf}.

\begin{theorem}\label{thm:finite}
Let $X$ be a two-sided shift space on a finite alphabet with language $\LLL$, and $\ph\colon X\to \RR$ a continuous potential. For each $x\in X$, define a Borel measure $\mmu_x$ on $\Wul(x)$ by \eqref{eqn:shift-mxu} and a Borel measure $\ms_x$ on $\Wsl(x)$ by \eqref{eqn:shift-mxs}. Then we have
\begin{align}\label{eqn:finite-u}
\uQ^+_x / \uQ &\leq \mmu_x(\Wul(x)) \leq \lQ^+_x, \\
\label{eqn:finite-s}
\uQ^-_x / \uQ &\leq \ms_x(\Wsl(x)) \leq \lQ^-_x.
\end{align}
\end{theorem}

Theorem \ref{thm:finite} guarantees that the measures $\mmu_x$ and $\ms_x$ are finite as soon as $\uQ<\infty$, but leaves open the possibility that they may vanish for some $x$. For a shift with specification and a potential with the Bowen property, there is $c>0$ such that $\lQ_x^{\pm} \geq c$ for all $x\in X$ \cite{rB745}, so that in particular all leaf measures are nonzero. See \S\ref{sec:spec} for a more general statement.

\subsection{Scaling under dynamics}\label{sec:dynamics}

A crucial property of $\alpha$-dimensional Hausdorff measure is the way it transforms under geometric scaling: if two sets are related by a transformation that scales distances by a factor of $r$, then their Hausdorff measures are related by a factor of $r^\alpha$. The corresponding property for the measures $\mmu_x$ and $\ms_x$ concerns how they transform under the dynamics of $\sigma$; they scale by a factor of $e^{P-\ph}$, as described in the following result, which is proved in \S\ref{sec:shift-mxs-pf}.

\begin{theorem}\label{thm:dynamics}
Let $X$ be a two-sided shift space on a finite alphabet with language $\LLL$, and $\ph\colon X\to \RR$ a continuous potential. Let $\mmu_x$ and $\ms_x$ be the families of leaf measures defined by \eqref{eqn:shift-mxu} and \eqref{eqn:shift-mxs}. Then if
$Z\subset \Wul(x)$ is a Borel set such that $\sigma(Z) \subset \Wul(\sigma x)$, we have
\begin{equation}\label{eqn:shift-mxu-scales}
\mmu_{\sigma x}(\sigma Z) = \int_Z e^{-\ph(y)+P} \,d\mmu_x(y).
\end{equation}
Similarly, given a Borel set $Z\subset \Wsl(x)$, we have
\begin{equation}\label{eqn:shift-mxs-scales}
\ms_{\sigma^{-1} x}(\sigma^{-1} Z) = \int_Z e^{-\ph(y)+P} \,d\ms_x(y).
\end{equation}
\end{theorem}

\begin{remark}\label{rmk:RN}
It is occasionally useful to write the scaling property in \eqref{eqn:shift-mxu-scales} in one of the following equivalent forms, valid whenever $\sigma y \in \Wul(\sigma x)$ (that is, $y_{(-\infty,1]} = x_{(-\infty,1]}$):
\begin{equation}\label{eqn:mxu-RN}
\frac{d\mmu_{\sigma x}}{d(\sigma_* \mmu_x)}(\sigma y) = e^{-\ph(y) + P}, 
\quad \text{or} \quad
\frac{d \mmu_x}{d(\sigma^*\mmu_{\sigma x})}(y) = e^{\ph(\sigma^{-1}y) - P}.
\end{equation}
If $y_{(-\infty,n]} = x_{(-\infty,n]}$, then iterating the second of these gives
\begin{equation}\label{eqn:mxu-RN-n}
\frac{d \mmu_x}{d((\sigma^n)^* \mmu_{\sigma^n x})}(y) = e^{\sum_{k=1}^{n} \ph(\sigma^{-k}y) - nP}.
\end{equation}
Analogous formulas hold for $\ms_x$.  In particular
\begin{equation}\label{eqn:mxs-RN-n}
\frac{d\ms_{\sigma^{-n}x}}{d((\sigma^n)_* \ms_x)}(y)
= e^{-\sum_{k=0}^{n-1}\ph(\sigma^{-k}y) + nP}.
\end{equation}
In \S\ref{sec:gibbs} we will use these scaling properties to obtain Gibbs-type estimates.
\end{remark}

\subsection{Scaling under holonomies}\label{sec:holonomies}

Recalling \S\ref{sec:prod}, we say $R \subset X$ has \emph{product structure}, and refer to $R$ as a \emph{rectangle}, if for every $x,y\in R$, the bracket $[x,y]$ from \eqref{eqn:lps} is defined and lies in $R$. 
Note that every rectangle must be contained in a 1-cylinder: if $R$ is a rectangle, then there is $a\in A$ such that $R\subset [a]^+ = [a]^-$. 

Given a rectangle $R$ and a point $x\in R$, we will write (see \eqref{eqn:WRq})
\[
\Wu_R(x) := \Wul(x) \cap R
\quad\text{and}\quad
\Ws_R(x) := \Wsl(x) \cap R.
\]
Given $z,y\in R$, the \emph{(stable) holonomy map} $\pi^s_{z,y} : \Wu_R(z) \to \Wu_R(y)$ is defined by
\begin{equation}\label{eqn:pi-u}
	\pi_{z,y}^s(x) = [y,x],
\end{equation}
so that $\{\pi_{z,y}^s(x)\} = \Wsl(x) \cap \Wul(y)$; the map $\pi^s_{z,y}$ ``slides'' a point $x$ along its stable leaf until it reaches $\Wul(y)$. We will be interested in rectangles on which the unstable leaf measures $\mmu_x$ transform nicely under holonomies.

\begin{definition}\label{def:cts-hol}
Given a collection $\RRR$ of rectangles, the family of unstable leaf measures has \emph{uniformly continuous holonomies} on $\RRR$ if for every $\eps>0$, there exists $\delta>0$ such that given any $R\in \RRR$ and $y,z\in R$ such that $d(y,z)<\delta$, we have $\mmu_y(Y) = e^{\pm\eps} \mmu_z(\pi_{y,z}^s(Y))$ for all measurable $Y\subset \Wu_R(y)$.\footnote{We use the notation $A = e^{\pm C} B$ to mean that $e^{-C}B\leq A \leq e^C B$.}
\end{definition}

We prove the following in \S\ref{sec:shift-cts-pf}.

\begin{proposition}\label{prop:open-hol}
If $\ph$ has the Walters property, 
then the leaf measures $\mmu_x$ have uniformly continuous holonomies on the family 
of open rectangles.
\end{proposition}

\begin{proposition}\label{prop:zero-hol}
If $\ph\equiv 0$, then the leaf measures $\mmu_x$ have uniformly continuous holonomies on the family of all rectangles.
\end{proposition}
\begin{remark}
A shift space has open rectangles if and only if it has a synchronizing word. For such shift spaces, Proposition \ref{prop:open-hol} allows us to study equilibrium measures for a broad class of potential functions. However, there are many shift spaces, such as the natural coding space for Sinai billiards \cite{BD20,CD-billiards}, in which we should not expect to find any open rectangles, and thus Proposition \ref{prop:open-hol} cannot be used, but Proposition \ref{prop:zero-hol} can still provide information about the measure of maximal entropy. We expect that the definition of the leaf measures could be modified to guarantee uniformly continuous holonomies on the family of all rectangles for a broader class of potentials, but we do not pursue this here.
\end{remark}

In the next section, it will be important to have a formula for the Radon--Nikodym derivative of the holonomy map. To give this, we start with the following direct consequence of \eqref{eqn:lps}: 
\begin{equation}\label{eqn:shift-bracket}
\text{if $x_{[0,n]} = y_{[0,n]}$, then $[\sigma^n x,\sigma^n y] = \sigma^n[x,y]$}.
\end{equation}
With \eqref{eqn:shift-bracket} in mind, the following is immediate.

\begin{lemma}\label{lem:shift-rectangle}
Given any rectangle $R$, any $n\in \NN$, and any $w\in \LLL_{n+1}$, the sets $\sigma^n(R\cap [w]^+) = \sigma^n(R) \cap [w]^-$ and $\sigma^{-n}(R\cap [w]^-) = \sigma^{-n}(R) \cap [w]^+$ are rectangles as well. (They may be empty.)
\end{lemma}

\begin{definition}\label{def:L-inv}
A family $\RRR$ of rectangles is \emph{$\LLL$-invariant} if given any $R\in \RRR$, $n\in \NN$, and $w\in \LLL_{n+1}$, the rectangles $\sigma^{n}(R\cap [w]^+)$ and $\sigma^{-n}(R \cap [w]^-)$ are contained in $\RRR$.
\end{definition}

The families of rectangles in Propositions \ref{prop:open-hol} and \ref{prop:zero-hol} are both $\LLL$-invariant.

\begin{definition}\label{def:pos-meas}
The family of leaf measures $\mmu_x$ is \emph{positive} on a family $\RRR$ of rectangles if for every $R\in \RRR$ there exists $x\in R$ such that $\mmu_x(\Wu_R(x))>0$.
\end{definition}

\begin{theorem}\label{thm:shift-cts}
Let $\RRR$ be an $\LLL$-invariant family of rectangles, and suppose that the family of leaf measures $\mmu_x$ is positive on $\RRR$. 
Then the following are equivalent.
\begin{enumerate}
\item The leaf measures $\mmu_x$ have uniformly continuous holonomies on $\RRR$.
\item The sum
\begin{equation}\label{eqn:Deltas}
	\Delta^s(x',x) := \sum_{n=0}^{\infty} \ph(\sigma^{n} x') - \ph(\sigma^{n} x)
\end{equation}
converges uniformly on $\{ (x,x') \in R^2 : R\in \mathcal{R}, x' \in \Ws_R(x)\}$,
is uniformly bounded on this set, and for every $R\in \RRR$ and $y,z\in R$, we have
\begin{equation}\label{eqn:shift-RN}
\frac{d(\pi^s_{z,y})_*\mmu_z}{d\mmu_y}(x)
= e^{\Delta^s((\pi^s_{z,y})^{-1}(x),x)}.
\end{equation}
\end{enumerate}
\end{theorem}

The proof for this result is in \S \ref{sec:shift-cts-pf}.  An analogous result is true if we reverse the roles of stable and unstable and change the sign of each instance of $n$. 

\begin{remark}\label{rmk:leafs-related}
It is worth observing that \eqref{eqn:shift-RN} gives
\[
	\mmu_z(\Wu_R(z)) = \mmu_z((\pi^s_{z,y})^{-1}(\Wu_R(y)))
	= \int_{\Wu_R(y)} e^{\Delta^s((\pi^s_{z,y})^{-1}(x),x)} \,d\mmu_y(x).
\]
Thus if the conditions of Theorem \ref{thm:shift-cts} are satisfied, we actually have a stronger version of positivity: $\mmu_y(\Wu_R(y))>0$ for every $y\in R\in \RRR$.
\end{remark}

\subsection{Product measure construction of an equilibrium measure}\label{sec:product-measure}

To go from the leaf measures $\mmu_x$ and $\ms_x$ to an equilibrium measure for $(X,\sigma,\ph)$, we start with a product construction. Suppose that $\RRR$ is an $\LLL$-invariant family of rectangles on which the families of leaf measures $\mmu_x$ and $\ms_x$ both have uniformly continuous holonomies and are positive in the sense of Definition \ref{def:pos-meas}.

Given $R\in \RRR$ and $q\in R$, consider the bijection $\iota_q\colon\Wu_R(q) \times \Ws_R(q) \to R$ from \eqref{eqn:iota}, defined by $\iota_q(x,y) = [y,x]$, and the product measure $m_{R,q} := (\iota_{q})_* (\mmu_{q} \times \ms_{q})$, for which
\begin{equation}\label{eqn:mxm}
m_{R,q}(Z)
= \int_{\Wu_R(q)}\int_{\Ws_R(q)}\one_{Z}([y,z])\,d\ms_{q}(y)\,dm^u_{q}(z).
\end{equation}
Let $\lambda_R \ll m_{R,q}$ be the measure on $R$ defined by
\begin{equation}\label{eqn:def-lambda-R}
\frac{d\lambda_R}{dm_{R,q}}(z) = 
e^{\Delta^u([y,z],y)} e^{\Delta^s([y,z],y)}.
\end{equation}
The following is proved in \S\ref{sec:es-two-sided} and justifies the suppression of $q$ in the notation $\lambda_R$.

\begin{theorem}\label{thm:q-indep}
Let $X$ be a two-sided shift space and $\ph\colon X\to \RR$ be continuous. Suppose that $\RRR$ is an $\LLL$-invariant family of rectangles on which the families of leaf measures $\mmu_x$ and $\ms_x$ both have uniformly continuous holonomies and are positive. Then for every $R\in \RRR$, the measure $\lambda_R$ defined in \eqref{eqn:mxm} and \eqref{eqn:def-lambda-R} is independent of the choice of $q\in R$, and is nonzero.
\end{theorem}

Given disjoint $R_1,\dots, R_\ell \in \RRR$, define a measure $\lambda_R$ on the union $R=\bigcup_{i=1}^\ell R_i$ by 
\begin{equation}\label{eqn:def-lambda}
	\lambda_R=\sum_{i=1}^\ell\lambda_{R_i}.
\end{equation}
The first return time function $\tau \colon R\to \NN \cup \{\infty\}$ is given by
\[
\tau(x) = \min \{n\geq 1 : \sigma^n(x) \in R\}.
\]
If $\lambda_R(\{x \in R : \tau(x)=\infty\}) = 0$, then the first return map $T \colon R\to R$ given by $T(x) = \sigma^{\tau(x)}(x)$ is defined $\lambda_R$-a.e. In \S\ref{sec:es-two-sided} we prove the following.

\begin{theorem}\label{thm:return-map}
Let $X,\ph,\RRR$ be as in Theorem \ref{thm:q-indep}, and $R,\tau,T$ as above. If $\lambda_R(\{x\in R : \tau(x) = \infty\}) = 0$, then $\lambda_R$ is $T$-invariant.
\end{theorem}

The procedure for passing from a $T$-invariant measure on $R$ to a $\sigma$-invariant measure on $X$ is standard: under the hypotheses of Theorem \ref{thm:return-map}, write $Y_n := \{y\in R : \tau(y) > n \}$, and consider the measure
\begin{equation}\label{eqn:es-mu}
\mu := \sum_{n= 0}^\infty\sigma_*^n \lambda_R|_{Y_n}.
\end{equation}
Clearly $\mu|_R = \lambda_R$. In \S\ref{sec:es-two-sided} we prove the following result.

\begin{theorem}\label{thm:push-product}
Let $X,\ph,\RRR,R,\tau,T$ be as in Theorems \ref{thm:q-indep} and \ref{thm:return-map}.
Suppose that $\lambda_R$ is positive, that $\int \tau \,d\lambda_R < \infty$, and that $\uQ<\infty$.
Then the measure $\mu$ defined in \eqref{eqn:es-mu} is positive, finite, and $\sigma$-invariant; moreover, its normalization $\mu/\mu(X)$ is an equilibrium measure with local product structure.
\end{theorem}

\subsection{Gibbs-type estimates}\label{sec:gibbs}

We conclude our general results with a description of how combining Theorem \ref{thm:finite} with the scaling estimates in \eqref{eqn:mxu-RN-n} provides Gibbs-type estimates on the weight that the leaf measures give cylinders. These lead in turn to Gibbs-type bounds for the measures $\lambda_R$ from \eqref{eqn:mxm}--\eqref{eqn:def-lambda-R}.

Given $x\in X$ and $w\in \LLL_n$ such that $w_1 = x_0$, we adopt the notation
\begin{equation}\label{eqn:x.w}
x.w := \cdots x_{-2} x_{-1} w_1 w_2 w_3 \cdots w_n.
\end{equation}
That is, $x.w$ is the left-infinite sequence obtained by taking the negative half of $x$ and appending $w$. (The result may or may not be legal in $X$: it is legal if and only if $\Wul(x) \cap [w]^+ \neq \emptyset$.) 
The sequence $x.w$ is indexed by $(-\infty,0] \cap \ZZ$, so $(x.w)_0 = w_n$, $(x.w)_{-1} = w_{n-1}$, and so on.
Now observe that for a symbol $a$ in the alphabet $A$,
\[
\sigma^n(\Wul(x) \cap [w]^+) = \bigcup_{a\in A} \Wul(x.wa),
\]
where as in Remark \ref{rmk:1-sided} we abuse notation slightly by using the notation $\Wul$ when $x.wa$ is only a one-sided infinite sequence; this is justified by Proposition \ref{prop:xy}. Some of the sets $\Wul(x.wa)$ may be empty.

Consider the quantity
\begin{equation}\label{eqn:Vxw}
V_x^+(w) := \sup \{ |S_n \ph(y) - S_n\ph(z)| : y,z \in \Wul(x) \cap [w]^+ \}.
\end{equation}
We see immediately that given any $y\in \Wul(x) \cap [w]^+$, we have 
\begin{equation}\label{eqn:Vxw-2}
\Phi_x^+(w) - V_x^+(w) \leq S_n\ph(y) \leq \Phi_x^+(w). 
\end{equation}
Given a single symbol $a\in A$ such that $\Wul(x.wa) \neq \emptyset$, write $\mmu_{x.wa}$ for the corresponding leaf measure; then \eqref{eqn:mxu-RN-n}, \eqref{eqn:Vxw-2}, and  Proposition \ref{prop:xy} give
\[
e^{\Phi_x^+(w) - nP - V_x^+(w)} \leq
\frac{d \mmu_x}{d((\sigma^*)^n \mmu_{x.wa})}(y)
\leq e^{\Phi_x^+(w) - nP}
\]
for each $y\in \Wul(x) \cap [wa]^+$. 
Writing $\|\mmu_{x.wa}\| := m_{x.wa}(\Wul(x.wa))$ for the total weight of the measure $\mmu_{x.wa}$, we obtain
\[
e^{-V_x^+(w)} \sum_a 
\|\mmu_{x.wa}\|
\leq \frac{\mmu_x([w]^+)}{e^{\Phi_x^+(w) - nP}}
\leq \sum_a \|\mmu_{x.wa}\|,
\]
and Theorem \ref{thm:finite} lets us conclude the following Gibbs-type estimate:

\begin{theorem}\label{thm:gibbs}
Let $X$ be a two-sided shift space on a finite alphabet $A$ with language $\LLL$, and $\ph\colon X\to \RR$ a continuous potential. Then given any $x\in X$ and $w\in \LLL_n$, we have
\begin{equation}\label{eqn:gibbs+}
\frac{\sum_{a\in A} \uQ^+_{x.wa}}{e^{V_x^+(w)} \uQ}
\leq  \frac{\mmu_x([w]^+)}{e^{\Phi_x^+(w) - nP}}
\leq  \sum_{a\in A} \lQ^+_{x.wa}.
\end{equation}
An analogous bound holds for $\ms_x([w]^-)$:
\begin{equation}\label{eqn:gibbs-}
	\frac{\sum_{a\in A}\uQ^-_{aw.x}}{e^{V_x^-(w)}\uQ}
	\leq \frac{\ms_x([w]^-)}{e^{\Phi_{x}^-(w)-nP}} \leq\sum_{a\in A}\lQ^-_{aw.x}.
\end{equation}
\end{theorem}

Once again we point out that if $X$ has the specification property and $\ph$ has the Bowen property, then the quantities in \eqref{eqn:gibbs+} and \eqref{eqn:gibbs-} are all bounded away from $0$ and $\infty$ independently of $x$ and $w$, yielding the usual Gibbs property.

Theorem \ref{thm:gibbs} yields an upper Gibbs bound for the product measures used in \S\ref{sec:product-measure}; the constant involved is given in terms of $\lQ$, and so for this bound to be useful we will need to know that $\lQ<\infty$.

\begin{corollary}\label{cor:lambda-Gibbs}
Let $X$ be a two-sided shift space on a finite alphabet and $\ph\colon X\to \RR$ be continuous. Suppose that $\RRR$ is an $\LLL$-invariant family of rectangles on which the families of leaf measures $\mmu_x$ and $\ms_x$ both have uniformly continuous holonomies.
Let $C = \max(\|\Delta^s\|,\|\Delta^u\|)$, which is finite by Theorem \ref{thm:shift-cts}, and let $K = (\#A) \lQ^2 e^{2C}$.
Then for every $R\in \RRR$, the measure $\lambda_R$ defined in \eqref{eqn:mxm} and \eqref{eqn:def-lambda-R} satisfies the following bound for all $w\in \LLL$:
\begin{equation}\label{eqn:lambda-Gibbs}
\lambda_R([w]^+) \leq K e^{\Phi(w) - |w|P}.
\end{equation}
\end{corollary}
\begin{proof}
If $[w]^+\cap R = \emptyset$, there is nothing to prove, so assume the intersection is nonempty.
By Theorem \ref{thm:q-indep}, we can define $\lambda_R$ using $q\in [w]^+ \cap R$. We get
	\begin{align*}
		\lambda_{R}([w]^+)
		&=\int_{\Wu_{R}(q)}\int_{\Ws_{R}(q)}e^{\Delta^u([y,z],y)}e^{\Delta^s([y,z],z)}\one_{[w]^+}\,d\ms_{q}(y)d\mmu_{q}(z)\\
		&\leq e^{2C}\int_{\Wul(q)}\int_{\Wsl(q)}\one_{[w]^+}\,d\ms_{q}(y)d\mmu_{q}(z)\\
		& = e^{2C}\int_{\Wul(q)\cap [w]^+}\int_{\Wsl(q)}\,d\ms_{q}(y)d\mmu_{q}(z)\\
		& = e^{2C}\mmu_{q}([w]^+)\ms_{q}(\Wsl(q)).
	\end{align*}
By Theorems \ref{thm:finite} and \ref{thm:gibbs} we have
\[
\mmu_{q}([w]^+)\ms_{q}(\Wsl(q))\leq  (\#A)\lQ^2e^{\Phi^+(w)},
\]
which proves the corollary.
\end{proof}

\section{Examples}\label{sec:spec}

\subsection{Subshifts of finite type}

Let $X$ be a topologically mixing subshift of finite type on a finite alphabet $A$, determined by a finite set $\FFF$ of forbidden words. Let $\ph\colon X\to\RR$ have the Walters property. Then it is well-known \cite{rB745,pW78} that $(X,\ph)$ has a unique equilibrium measure, and that $\uQ<\infty$ and $\inf_x \lQ^+>0$, $\inf_x \lQ^->0$, so the leaf measures $\mmu_x$ and $\ms_x$ are uniformly positive and finite by Theorem \ref{thm:finite}.

The product measure construction in \S\ref{sec:product-measure} can be carried out as follows.
Let $k=\max\{|u|:u\in\FFF\}$, and observe that for all $w\in\LLL$ with $|w|\geq k$, the cylinder $[w]^+$ has product structure: indeed, given any $x,y \in [w]^+$, any subword of $[x,y]$ with length $\leq k$ must be a subword of either $x$ or $y$, and thus cannot lie in $\FFF$ (since $x,y\in X$), allowing us to conclude that $[x,y] \in X$ and hence $[x,y]\in [w]^+$.  

By Proposition \ref{prop:open-hol}, the leaf measures $\mmu_x$ and $\ms_x$ have uniformly continuous holonomies on the family $\RRR$ of open rectangles. By the previous paragraph, this family contains $[w]^+$ for every $w\in \LLL_k$. So we can write $X$ as the finite union of open rectangles:
\[X=\bigcup_{w\in\LLL_k}[w]^+.\]
On each $[w]^+$, \eqref{eqn:mxm} and \eqref{eqn:def-lambda-R} produce a positive and finite measure $\lambda_w$, and taking $R = X$ we have return time $\tau\equiv 1$, so the construction in \eqref{eqn:def-lambda} and \eqref{eqn:es-mu} reduces to
\[
\mu = \sum_{w\in \LLL_k} \lambda_w.
\]
The return time is integrable (since it is bounded and each $\lambda_w$ is finite), so by Theorem \ref{thm:push-product}, $\mu/\mu(X)$ is the unique equilibrium measure.

\subsection{Synchronizing words}\label{sec:synchronizing-words}

The result for subshifts of finite type can be extended to shift spaces with a synchronizing word and the appropriate counting estimates.  

\begin{definition}\label{def:sync}
A word $v$ is said to be \emph{synchronizing} if for every pair $u,w\in\LLL$ such that $uv,vw\in\LLL$, we have that $uvw\in\LLL$ as well.
\end{definition}

Equivalently, $v$ is synchronizing if the cylinder $[v]^+$ has product structure; in particular, $X$ has a synchronizing word if and only if the family $\RRR$ of open rectangles (as in Proposition \ref{prop:open-hol}) is nonempty. In this section we give a condition for our main results to apply to such shift spaces. In a separate paper \cite{CD-billiards}, we will study shift spaces that code dispersing billiards and do not have any open rectangles.

We remark that a shift space is a subshift of finite type if and only if all sufficiently long words are synchronizing. To go beyond SFTs, we impose a condition on sequences that do not contain a given synchronizing word: given a synchronizing word $v\in \LLL$ for a shift space $X$, we write
\begin{equation}\label{eqn:Yv}
Y_v = \{x\in X : x \text{ does not contain } v\}.
\end{equation}
Observe that $Y_v$ is closed and shift-invariant.

\begin{theorem}\label{thm:sync}
Let $X$ be a shift space with a synchronizing word $v$. Suppose that $\ph\colon X\to \RR$ has the Walters property, and that the subshift $Y_v$ defined in \eqref{eqn:Yv} satisfies
\begin{equation}\label{eqn:sync-assumption}
P(Y_v,\ph) < P(X,\ph).
\end{equation}

Then the following are true.
\begin{enumerate}
\item $(X,\sigma,\ph)$ has a unique equilibrium measure.\label{item:uniquness}
\item The equilibrium measure has local product structure in the sense of \S\ref{sec:prod}.
\item Writing $R = [v]^+$ and letting $\tau \colon R\to \NN \cup \{\infty\}$ be the first return time,
the measure $\lambda = \lambda_R$ defined by \eqref{eqn:mxm}--\eqref{eqn:def-lambda-R}
is nonzero and $\int \tau \,d\lambda < \infty$.
\item The $\sigma$-invariant measure $\mu$ defined in \eqref{eqn:es-mu} is a scalar multiple of the unique equilibrium measure.
\item There exists $K>0$ such that if $w\in \LLL$ begins and ends with the synchronizing word $v$, then 
\begin{equation}\label{eqn:sync-Gibbs}
K^{-1} e^{-|w|P(X,\ph) + \Phi(w)} \leq \mu([w]^+) \leq K e^{-|w|P(X,\ph) + \Phi(w)}.
\end{equation}
\end{enumerate}
\end{theorem}

\begin{remark}\label{rmk:spec}
To deduce Theorem \ref{thm:spec} from Theorem \ref{thm:sync}, it suffices to observe that if $X$ is a shift space with the specification property, then it has a synchronizing word \cite{aB88}, and \eqref{eqn:sync-assumption} holds for every potential with the Bowen property.  This last assertion can be deduced using the fact that $(X,\ph)$ has a unique equilibrium measure $\mu$, and that $\mu$ is fully supported \cite{rB745}, so that in particular $\mu([v]^+)>0$ and therefore $\mu(Y_v)=0$. 
By upper semi-continuity of entropy, $(Y_v,\ph)$ has an equilibrium measure $\nu \neq \mu$, for which $P(Y_v,\ph) = h(\nu) + \int \ph\,d\nu < P(X,\ph)$. 
\end{remark}

The remainder of this section is devoted to the proof of Theorem \ref{thm:sync}, which will proceed along the following lines.
\begin{enumerate}[wide=0pt, leftmargin=*, itemsep=1ex, label={\textsc{Step} \arabic{*}:}]
\item Use the synchronizing property together with \eqref{eqn:sync-assumption} to prove that the set $\GGG \subset \LLL$ of words that start and end with $v$ has the specification property.
\item Use this specification property and \eqref{eqn:sync-assumption} to apply results from \cite{CT13,PYY22}, deducing uniqueness and uniform counting bounds on partition sums.
\item Use these counting bounds together with Theorem \ref{thm:finite} to obtain lower bounds on $\mmu_x$ and $\ms_x$ whenever $x\in [v]^+$, thus deducing that $\lambda([v]^+)>0$.
\item Use \eqref{eqn:sync-assumption}  and the Gibbs bounds from \S\ref{sec:gibbs} to show that $\int \tau\,d\lambda<\infty$ and to deduce the Gibbs bounds in \eqref{eqn:sync-Gibbs} for every $w\in \GGG$.
\item Apply Theorem \ref{thm:push-product} to deduce that $\mu/\mu(X)$ is an equilibrium measure.
\end{enumerate}

Before carrying these out, we set up some notation and recall a lemma. Recalling the notation $\Lambda_n=\sum_{|w|=n}e^{\Phi(w)}$ from \eqref{eqn:Lambda}, given $\DDD\subset \LLL$ we write
\[
\Lambda_n(\DDD)=\sum_{w\in\DDD_n}e^{\Phi(w)}
\quad\text{and}\quad
Q_n(\DDD) = \Lambda_n(\DDD) e^{-nP(X,\ph)}.
\]
A collection $\DDD \subset \LLL$ is called \emph{factorial} if it is closed under passing to subwords. The subshift associated to a factorial collection $\DDD$ is
\[
X_\DDD := \{x\in X : \text{every subword of $x$ is in }\DDD\}.
\]
\begin{lemma}\label{lem:factorial}
If $\DDD$ is factorial and $\ph$ is continuous, then $\frac 1n \log \Lambda_n(\DDD) \to P(X_\DDD,\ph)$.
\end{lemma}
\begin{proof}
For $\ph=0$, this is \cite[Lemma 2.7]{CP19}; the proof in \S4.4 of that paper extends to arbitrary continuous $\ph$ with routine modifications; see Appendix \ref{app:pressure}.
\end{proof}

We will apply Lemma \ref{lem:factorial} to the factorial collection
\[
\BBB := \{w\in \LLL : w \text{ does not contain }v\},
\]
for which we have $X_\BBB = Y_v$, so that the condition $P(Y_v,\ph) < P(X,\ph)$ together with Lemma \ref{lem:factorial} guarantees that
\begin{equation}\label{eqn:PB-gap}
\lim_{n\to\infty} \frac 1n \log \Lambda_n(\BBB) < P(X,\ph).
\end{equation}
We will often use the following consequence of this:
\begin{equation}\label{eqn:summable}
\sum_{n=1}^\infty Q_n(\BBB)
= \sum_{n=1}^\infty \Lambda_n(\BBB) e^{-nP(X,\ph)}
<\infty.
\end{equation}
With these preliminaries set up, we turn our attention to Step 1 of the proof of Theorem \ref{thm:sync}. We need the following lemma.

\begin{lemma}\label{lem:sync-spec}
	If equation \eqref{eqn:summable} holds, then there exists $u\in\LLL$ such that $vuv\in\LLL$.
\end{lemma}
\begin{proof}
Let $\DDD = \BBB \cup ((\BBB v \BBB) \cap \LLL)$ be the set of words in $\LLL$ that do not contain two disjoint copies of $v$. Then we have
\begin{equation}\label{eqn:LnD-leq}
\Lambda_n(\DDD) \leq \Lambda_n(\BBB)+ e^{\Phi(v)}\sum_{i=0}^{n-|v|} \Lambda_i(\BBB) \Lambda_{n-i-|v|}(\BBB).
\end{equation}
Recall also from \S\ref{sec:finite} that $\Lambda_n = \Lambda_n(\LLL) \geq e^{nP}$, so \eqref{eqn:LnD-leq} gives
\begin{equation}\label{eqn:LnLn}
\frac{\Lambda_n(\DDD)}{\Lambda_n(\LLL)} \leq \Lambda_n(\DDD) e^{-nP}
\leq Q_n(\BBB)+e^{\Phi(v)-|v|P}
\sum_{i=0}^{n-|v|} 
Q_i(\BBB)Q_{n-|v|-i}(\BBB).
\end{equation}
Writing $c_i = Q_i(\BBB)$, observe that \eqref{eqn:summable} gives
\begin{equation}\label{eqn:ci-trick}
\sum_{\ell=0}^\infty \sum_{i=0}^\ell c_i c_{\ell-i} = \Big( \sum_{k=0}^\infty c_k \Big)^2 < \infty,
\end{equation}
so $\sum_{i=0}^\ell c_i c_{\ell-i} \to 0$ as $\ell\to \infty$.  Thus the right-hand side of \eqref{eqn:LnLn} converges to $0$ as $n\to\infty$. For sufficiently large $n$ this gives $\Lambda_n(\DDD) < \Lambda_n(\LLL)$, so $\DDD \neq \LLL$, which proves Lemma \ref{lem:sync-spec}.
\end{proof}

\begin{proposition}\label{prop:decomp-spec}
Writing $\GGG = \LLL \cap v\LLL \cap \LLL v$ for the set of words in $\LLL$ that both start and end with $v$, the following are true whenever \eqref{eqn:summable} holds.
\begin{itemize}
\item $\BBB\GGG\BBB$ is a decomposition of the language $\LLL$: for every $w\in \LLL$ there are $u^p, u^s \in \BBB$ and $u^g \in \GGG$ such that $w = u^p u^g u^s$.
\item $\GGG$ has the specification property: there is $t\in \NN$ such that for every $w,w'\in \GGG$, there is $u\in \LLL_t$ such that $wuw'\in \GGG$.
\end{itemize}
\end{proposition}
\begin{proof}
For the first claim, if $w$ does not contain $v$ then we take $u^p = w$ and set $u^g$ and $u^s$ to be the empty word. If $w$ does contain $v$, then we take $u^p$ to be the initial segment of $w$ preceding the first appearance of $v$, we take $u^s$ to be the final segment of $w$ following the last appearance of $v$, and we take $u^g$ to be the subword of $w$ lying between $u^p$ and $u^s$.

For the second claim, let $u\in \LLL$ be provided by Lemma \ref{lem:sync-spec} and let $t = |u|$. Then given any $w,w'\in \GGG$, we claim that $wuw'\in \GGG$. Indeed, by the definition of $\GGG$ we have $w = qv$ and $w' = vq'$ for some $q,q'\in \LLL$, and thus $wuw' = qvuvq'$. By the synchronizing property, since $qv\in \LLL$ and $vuv\in \LLL$ we deduce that $qvuv\in \LLL$, and since $vq'\in \LLL$ we apply the property again to get $wuw' = qvuvq'\in \LLL$. This word starts and ends with $v$ since $w$ and $w'$ do, so it lies in $\GGG$.
\end{proof}

Now we are in a position to carry out Step 2 and obtain uniqueness and uniform counting bounds. These are provided by the following result.

\begin{theorem}\label{thm:CT-PYY}
Suppose $X$ is a shift space on a finite alphabet with language $\LLL$, and that $\ph\colon X\to \RR$ is continuous. Suppose moreover that there are $\CCC^p, \GGG, \CCC^s \subset \LLL$ such that the following conditions hold.
\begin{enumerate}[label=\upshape{(\Roman{*})}]
\item\label{spec} Specification on $\GGG$: there is $t\in\NN$ such that for all $w^1,w^2,\dots,w^n\in \GGG$, there are $u^1,\dots,u^{n-1}\in \LLL_{t}$ such that $w^1 u^1 w^2 u^2 \cdots u^{n-1} w^n\in\LLL$.
\item\label{Bow} Bowen property on $\GGG$: there is $C>0$ such that for every $w\in \GGG$ and $x,y \in [w]$, we have $|S_{|w|}\ph(x) - S_{|w|}\ph(y)| \leq C$.
\item\label{gap} $\ulim_{n\to\infty} \frac 1n \log \Lambda_n(\CCC^p \cup \CCC^s) < P(\ph)$.
\end{enumerate}
Then $X$ has a unique equilibrium measure, and $\uQ<\infty$.
\end{theorem}
\begin{proof}
This is very nearly the form of \cite[Theorem C]{CT13}, except that there, the specification property is also required to hold for the sets $\GGG(M)=\{u^pwu^s\in\LLL : u^p \in \CCC^p$, $u^s \in \CCC^s$, $w\in\GGG$, $|u^{p,s}|\leq M\}$. An analogous result for flows was proved in \cite{CT16}, and was later improved by Pacifico, Yang, and Yang \cite{PYY22}, who showed that it suffices to have specification for $\GGG$ itself. Their arguments can be readily adapted to the symbolic setting, see \cite{blog}, but since the precise statement given in Theorem \ref{thm:CT-PYY} cannot be directly read off from existing results in the literature, we provide in Appendix \ref{app:unique} a proof of Theorem \ref{thm:CT-PYY} as a logical consequence of \cite[Theorem A]{PYY22} by using suspension flows.
\end{proof}

In our setting, we have $\CCC^p = \CCC^s = \BBB$, and then Proposition \ref{prop:decomp-spec} provides Condition \ref{spec}, the Walters property implies Condition \ref{Bow}, and \eqref{eqn:sync-assumption} provides Condition \ref{gap}, so Theorem \ref{thm:CT-PYY} applies and we can conclude that there is a unique equilibrium measure, and that $\uQ<\infty$.

We will also need uniform lower counting bounds on $\GGG$; these follow from similar arguments to those in \cite{CT13}, but cannot be deduced directly from what is written there, so we provide a proof.

\begin{proposition}\label{prop:lambda-snyc-pos}
If equation \eqref{eqn:summable} holds, then $\uQ(\GGG) := \ulim_{n\to\infty} Q_n(\GGG) > 0$.
\end{proposition}
\begin{proof}
Using the decomposition $\BBB\GGG\BBB$ provided by Proposition \ref{prop:decomp-spec}, we have
\begin{equation}\label{eqn:QnL}
Q_n(\LLL) \leq \sum_{i+j+k=n} Q_i(\BBB)Q_j(\GGG)Q_k(\BBB)
 =\sum_{\ell=0}^n Q_{n-\ell}(\GGG) \sum_{i=0}^\ell Q_i(\BBB)Q_{\ell-i}(\BBB).
\end{equation}
Let $R_\ell(\BBB) := \sum_{i=0}^\ell Q_i(\BBB) Q_{\ell-i}(\BBB)$, and observe that just as in \eqref{eqn:ci-trick}, we have
\begin{equation}\label{eqn:Rell}
\sum_{\ell=0}^\infty R_\ell(\BBB) = \Big( \sum_{i=0}^\infty Q_i(\BBB) \Big)^2 < \infty.
\end{equation}
Thus $S := \max_\ell R_\ell(\BBB) < \infty$. Since $\uQ<\infty$, we have $K := \max_j Q_j(\GGG) < \infty$. By \eqref{eqn:Rell}, there is $N\in \NN$ such that $\sum_{\ell=N}^\infty R_\ell(\BBB) < 1/(2K)$, so if $n>N$, using \eqref{eqn:QnL} and recalling \eqref{eqn:Ln-geq}, we obtain
\begin{align*}
1 &\leq Q_n(\LLL) \leq \sum_{\ell=0}^n Q_{n-\ell}(\GGG) R_\ell(\BBB)
\leq \sum_{\ell=0}^{N-1} Q_{n-\ell}(\GGG) S + \sum_{\ell=N}^n K R_\ell(\BBB) \\
&\leq SN \max \{Q_j(\GGG) : n-N < j \leq n \} + K/(2K).
\end{align*}
The last term is equal to $\frac 12$; subtracting it from both sides and dividing by $SN$ gives
\[
\max\{Q_j(\GGG) : n-N < j \leq n \} \geq 1/(2SN).
\]
Since $n > N$ was arbitrary this proves Proposition \ref{prop:lambda-snyc-pos}.
\end{proof}

We have now completed Step 2 of the proof, and proceed to Step 3, the task of obtaining lower bounds on $\mmu_x$ and $\ms_x$ when $x\in [v]^+$. With Proposition \ref{prop:lambda-snyc-pos} in hand, we can obtain these using Theorem \ref{thm:finite}; to show that both of these measures are nonzero, it suffices to show that $\uQ_x^{\pm}>0$ for all $x\in [v]^+$.

Fix $x\in [v]^+$, and let $u\in \LLL$ be as in Lemma \ref{lem:sync-spec} (so $vuv\in \LLL$). 
Given $w\in \GGG$, let $y_w$ be the periodic sequence $\cdots uwu.wuwu\cdots$, and let $z_w = [x,y_w] \in \Wul(x) \cap [w]^+$. Thus
\begin{equation}\label{eqn:Phix}
\Phi^+_x(w) \geq S_{|w|} \ph(z_w) \geq \Phi(w) - C,
\end{equation}
where $C$ is the constant from the Bowen property.
From the definitions of $\Lambda_n^+(x)$ and $\uQ_x^+$ in \eqref{eqn:Lambda+} and \eqref{eqn:Q+} we now have
\[
\uQ_x^+ = \ulim_{n\to\infty} \sum_{w\in \LLL_n} e^{\Phi_x^+(w) - nP}
\geq \ulim_{n\to\infty} \sum_{w\in \GGG_n} e^{\Phi(w) - C - nP} = e^{-C} \uQ(\GGG) > 0.
\]
By Theorem \ref{thm:finite}, this implies that $\mmu_x(\Wul(x)) \geq e^{-C} \uQ(\GGG)/\uQ > 0$. 

The argument for $\ms_x$ is similar: fix $x\in [v]^+$ and given $w\in \GGG$, let $z_w = [y_w,x] \in \Wsl(x) \cap [wu]^-$. In place of \eqref{eqn:Phix} we have
\[
\Phi_x^-(wu) \geq S_{|wu|}\ph(\sigma^{-|wu|}(z_w)) \geq \Phi(wu) - C,
\]
and we get
\begin{align*}
\uQ_x^- &\geq \ulim_{n\to\infty} \sum_{w\in \GGG_n} e^{\Phi(wu) - C - (n+|u|)P} \\
&\geq \ulim_{n\to\infty} \sum_{w\in \GGG_n} e^{\Phi(w) + \Phi(u) - 2C - nP - |u|P} 
\geq e^{\Phi(u) - 2C - |u|P} \uQ(\GGG) > 0.
\end{align*}
This completes Step 3 of the proof; we have now shown that $\lambda([v]^+)=0$.

Proceeding to Step 4, we show integrability of the return time by using \eqref{eqn:summable} and the Gibbs bound \eqref{eqn:lambda-Gibbs} from Corollary \ref{cor:lambda-Gibbs}; note that the constant $K = (\#A) e^{2C} \lQ^2$ appearing there is finite because $\uQ<\infty$.
Writing $Y_n=\{x\in [v]^+:\tau(x)>n\}$, we observe that
\[
Y_n \subset \bigcup_{u\in \BBB_n} [vu]^+
\]
(note that the cylinder is empty if $vu\notin \LLL$),
and thus \eqref{eqn:lambda-Gibbs} gives
\begin{align*}
\lambda(Y_n) &\leq \sum_{u\in \BBB_n} \lambda([vu]^+)
\leq \sum_{u\in \BBB_n} K e^{\Phi(vu) - |vu|P} \\
&\leq \sum_{u\in\BBB_n} K e^{\Phi(v) + \Phi(u) - |v|P - |u|P}
= e^{\Phi(v) - |v|P} Q_n(\BBB).
\end{align*}
Thus the integral of the return time is
\[
\int \tau\,d\lambda = \sum_{n=0}^\infty \lambda(Y_n) \leq e^{\Phi(v) - |v|P} \sum_{n=0}^\infty Q_n(\BBB) < \infty.
\]

Moreover, given any word $w\in \LLL$ that begins and ends with $v$, we have $[w]^+ \subset [v]^+ = R$, so $\mu([w]^+) = \lambda([w]^+)$, and Corollary \ref{cor:lambda-Gibbs} implies the second inequality in \eqref{eqn:sync-Gibbs}. The first inequality in \eqref{eqn:sync-Gibbs} follows from \eqref{eqn:gibbs+} in Theorem \ref{thm:gibbs}, the specification property for $\GGG$ and the Walters property for $\ph$, and the counting bounds in Proposition \ref{prop:lambda-snyc-pos}.  This completes Step 4.

For the completion of the proof, Step 5, it suffices to check the conditions of Theorem \ref{thm:push-product}. We showed in Step 3 that $\lambda$ is positive, in Step 4 that $\int\tau\,d\lambda<\infty$, and in Step 2 that $\uQ<\infty$. Thus Theorem \ref{thm:push-product} applies, showing that the measure $\mu$ defined in \eqref{eqn:es-mu} is positive, finite, $\sigma$-invariant, and its normalization is an equilibrium measure with local product structure. Since we proved earlier that there is a unique equilibrium measure, this completes the proof of Theorem \ref{thm:sync}.

\section{Weights for leaf measures in Theorem \ref{thm:finite}}\label{sec:shift-mxu-pf}

Observe that for any potential $\ph$, if $\bar{\ph}=\ph-P(\ph)$, then $P(\bar{\ph})=0$.  Moreover, under this construction $\ph$ and $\bar{\ph}$ produce the same equilibrium measure as 
\[\Phi^+_x(w)-|w|P(\ph)=\bar\Phi^+_x(w)=\sup\{S_{|w|}\bar\ph(y):y\in[w]^+\cap\Wul(x)\}.\]  
To simplify the notation in the proofs, we will replace $\ph$ with $\bar{\ph}$ and make the following assumption for the remainder of the paper.
\begin{equation}\label{eqn:0-P-assumption}
	\textbf{Standing Assumption: }P=0. 
\end{equation}
For nonzero pressure, the work is similar where an additional term including the pressure would be included.  For example, wherever a $\Phi^{\pm}_x(w)$ or $\ph(x)$ appears, one would need to include $\Phi^{\pm}_x(w)-|w|P$ or $\ph(x)-P$, respectively.

Equations \eqref{eqn:finite-u} and \eqref{eqn:finite-s} are equivalent results.  By proving one result, one can prove the other by reversing the direction of the shift.  In this section we will prove the bounds on the weights of the unstable leaf measure in \eqref{eqn:finite-u}.  

To obtain these bounds, we first need to establish a pair of inequalities, which should be compared to \eqref{eqn:almost-add}. First, recall from \eqref{eqn:x.w} that for $x\in X$ and $w\in\LLL_n$ with $[w]^+ \cap \Wul(x) \neq\emptyset$, we write
$\xminus w=\dots x_{-2} x_{-1} w_1 w_2\dots w_n$,
and define
\[\Phi_{\xminus w}(u)=\sup\{S_{|u|}\ph(y): y\in \Wul(\xminus wu_1)\cap[u]^+\}.\]

\begin{lemma}\label{lem:2-Phi-bounds}  
	For $w,u,v\in\LLL$ such that $wu\in\LLL$
	and $uv\in\LLL$
	\begin{align}
		\Phi^+_{x}(wu)&\leq \Phi^+_{x}(w)+\Phi_{\xminus wu_1}^+(u)\label{eqn:upper-add-mxu}\\
		\Phi^+_{x}(uv)&\geq \Phi^+_{x}(u)+\Phi_{\xminus uv_1}^+(v)-V_x^+(u) \label{eqn:lower-add-mxu}
	\end{align}
	where $V^+_x(u)=\sup_{y,z\in \Wul(x)\cap[u]^+}(S_{| u|}\ph(y)-S_{| u|}\ph(z))$. 
\end{lemma}

\begin{proof}
	Let $y\in\Wul(x)\cap [wu]^+$ and observe that
	\[S_{|wu|}\ph(y)=S_{|w|}\ph(y)+S_{| u|}\ph(\sigma^{| w| }y).\]
	If $y'\in\Wul(x)\cap [w]^+$ and $z\in \Wul(x.wu_1)\cap [u]^+$, then taking supremums yields
	\begin{align*}
		\Phi_x^+(wu)& =\sup_{y}S_{|wu|}\ph(y)
		\leq \sup_{y',z}S_{|w|}\ph(y')+S_{| u|}\ph(\sigma^{|w|}z)\\
		& =\Phi_{x}^+(w)+\Phi_{\xminus wu_1}^+(u).
	\end{align*}
	Similarly, if $y \in [uv]^+$ and $y'\in \Wul(x)\cap[u]^+$ we have
	\begin{align*}
		S_{|u|}\ph(y)+S_{|v|}\ph(\sigma^{|v| }y)&=S_{|u|}\ph(y)-S_{|u|}\ph(y') + S_{|u|}\ph(y')+S_{|v|}\ph(\sigma^{|u| }y)\\
		&\geq -V_x^+(u)+ S_{|u|}\ph(y')+S_{|v|}\ph(\sigma^{|u|}y).
	\end{align*}
	If $z\in\Wul(x.uv_1)\cap [v]^+$, then taking supremums gives us
	\begin{align*}
		\Phi_x^+(uv)&=\sup_{y}S_{|u|}\ph(y)+S_{|v|}\ph(\sigma^{|v|}y) 
		\geq \sup_{y',z}-V_x^+(u)+ S_{|u|}\ph(y')+S_{|v|}\ph(z)\\
& = \Phi_{x}^+(u)+\Phi_{\xminus uv_1}(v)-V_x^+(u).
	\end{align*}
	Since every $z\in\Wul(x.uv_1)\cap[v]^+$ is equal to $\sigma^{|u|}y$ for some $y\in\Wul(x)\cap [uv]^+$.
\end{proof}

Now we prove the inequalities in Theorem \ref{thm:finite} for unstable leaves.

\begin{proof}
	We begin with the proof of the upper bound in \eqref{eqn:finite-u}.  For all $n\in\NN$, we have $\Wul(x) \subset X = \bigcup_{u\in \LLL_n} [u]^+$, so $\LLL_n \in \EE^+(\Wul(x),N)$ whenever $n\geq N$. Thus
	\begin{equation}\label{eqn:2-Phi-upper}
		\inf\left\{\sum_{u\in\EEE}e^{\Phi^+_{x }(u)}:\EEE\in\EE^+(\Wul(x),N)\right\}\leq \sum_{u\in\LLL_{n}}e^{\Phi^+_{x}(u)} = \Lambda_n^+(x).
	\end{equation}	
	Sending $n\to\infty$ along an arbitrary subsequence and recalling \eqref{eqn:Q+}, this gives
	\[m_x^u(\Wul(x))\leq \lQ^+_{x},\]
	which proves the upper bound in equation \eqref{eqn:finite-u}.  
	
	For the lower bound in \eqref{eqn:finite-u}, we begin by finding an inequality for expressing a fixed cylinder as a disjoint union of longer cylinders of some uniform length.  Writing $\SSS_n(\xminus u):=\{v\in\LLL_n:[v]^+ \cap \Wul(x.u) \neq\emptyset\}$, then by inequality \eqref{eqn:upper-add-mxu}, we have
	\begin{align*}
		\sum_{v\in\SSS_n(\xminus u)}e^{\Phi^+_{x}(uv)} & \leq  \sum_{v\in\SSS_n(\xminus u)}e^{\Phi^+_{x}(u)+\Phi^+_{\xminus uv_1}(v)} 
		= e^{\Phi^+_{x}(u)} \sum_{v\in\SSS_n(\xminus u)}e^{\Phi^+_{\xminus uv_1}(v)} \\
		& \leq  e^{\Phi^+_{x}(u)}\sum_{v\in\LLL_n}e^{\Phi(v)}  
		 = e^{\Phi^+_{x}(u)}Q_n.
	\end{align*}
	This implies that 
	\begin{equation}\label{eqn:unfm-cyl-len}
		\left(Q_n\right)^{-1}\sum_{v\in\SSS_n(\xminus u)}e^{\Phi^+_{x}(uv)}\leq e^{\Phi^+_{x}(u)}.
	\end{equation}
	Next, we will take a cover of $\Wul(x)$ and lengthen the words of the cylinders so they are all the same length, then apply the bound in \eqref{eqn:unfm-cyl-len}.
	
By compactness of $\Wul(x)$, it suffices to consider only finite covers of $\Wul(x)$ in estimating $\mmu_x(\Wul(x))$.  Given such a finite cover $\EEE$ and any $n \geq \max \{ |u| : u\in \EEE\}$ (which is finite since $\EEE$ is), we produce a new cover $\{uv : u \in \EEE, v\in \SSS_{n-|u|}(x.u)\}$. All words of this new cover have length $n$.
Now \eqref{eqn:unfm-cyl-len} gives
	\begin{equation}\label{eqn:2-w-factor}
		\sum_{u\in\EEE} e^{\Phi^+_{x}(u)}\geq \sum_{u\in\EEE}\left(Q_{n-|u|}\right)^{-1}\sum_{v\in\SSS_{n-| u|}(\xminus u)}e^{\Phi^+_{x}(uv)}.
	\end{equation}
	We desire to find a bound for the $Q_{n-|u|}$ so that we can factor it out of the sum.  For all $Q>\uQ$ there exists an $n_0$ such that $Q>Q_k$ for all $k\geq n_0$.  If we restrict $n$ to be greater than $n_0+\max_{u\in\EEE}|u|$, then \eqref{eqn:2-w-factor} implies that
	\begin{equation}\label{eqn:2-w-arb}
			\sum_{u\in\EEE}e^{\Phi^+_{x}(u)}\geq {Q}^{-1}\sum_{u\in\EEE}\sum_{v\in\SSS_{n-|u|}(\xminus u)}e^{\Phi^+_{x}(uv)}
			\geq{Q}^{-1}\sum_{w\in\LLL_n}e^{\Phi^+_{x}(w)}.
	\end{equation}
	The new index in the last step follows from the fact that for each $u$, the word $uv$ has length $n$, and the collection of $[uv]^+$ covers $\Wul(x)$.  Additionally, any word $w\in\LLL_n$ such that $\Wul(x)\cap [w]^+=\emptyset$ does not contribute to the sum. 
	
	By taking an arbitrary subsequence $n_k\to\infty$ in \eqref{eqn:2-w-arb} we have 
	\[m^u_x(\Wul(x))\geq{Q}^{-1}\uQ^+_{x}.\]
	Since $Q>\uQ$ was arbitrary, this proves the lower bound in \eqref{eqn:finite-u}.
\end{proof}
The proof for \eqref{eqn:finite-s} is analogous. 
\section{Scaling of leaf measures in Theorem \ref{thm:dynamics}}\label{sec:shift-mxs-pf}

We will prove the scaling property in Theorem \ref{thm:dynamics} for unstable leaves by proving a preliminary lemma.  The approach for the stable setting is identical by reversing the direction of the shift.

We will use the notation $\ph(x)=c\pm\eps$ as a shorthand for $c-\eps\leq \ph(x) \leq c+\eps$.

\begin{lemma}\label{lem:shift-mxu-scales}
	Let $Z\subset \Wul(x)$ such that $\sigma Z\subset \Wul(\sigma x)$.  Assume further that there is a constant $c\in\RR$ and $N_0\in\NN$ such that for any $w\in\LLL_{\geq N_0}$ with $[w]^+\cap Z\neq\emptyset$ we have that $\ph(y)=c\pm\eps$ for all $y\in[w]^+ \cap \Wul(x)$. Then we have
	\[\mmu_{\sigma x}(\sigma Z)=e^{\pm2\eps}\int_Ze^{-\ph(y)}\,d\mmu_x(y).\]
\end{lemma}
\begin{proof}
	If $|w|\geq N_0$ we have that $\ph(x)=c\pm\eps$ on $[w]^+$ so
	\begin{equation}
		\begin{split}\label{eqn:Phi-plus-cont}
			\Phi_x^+(w)&=\sup\left\{S_{|w|}\ph(y): y\in [w]^+\cap \Wul(x)\right\}\\
			& = \sup\left\{S_{|w|-1}\ph\left(\sigma y\right)+c\pm\eps: y\in [w]^+\cap \Wul(x)\right\}\\
			&=c\pm\eps+\sup\left\{S_{|w|-1}\ph(y'): y'\in [\sigma w]^+\cap \Wul(\sigma x)\right\}\\
			&=c\pm\eps+\Phi_{\sigma x}^+(\sigma w).
		\end{split}
	\end{equation}
	Because $\sigma(Z)\subset\Wul(\sigma x)$ for any $z\in Z$, the symbol $z_{1}=x_{1}$.	This implies that for a sufficiently long word $w$ such that $[w]^+\cap Z$, the map $\sigma$ is a bijection between $[w]^+\cap\Wul(x)$ and $[\sigma w]^+\cap \Wul(\sigma x)$.
	
	Moreover, $\sigma$ acts as a bijection between covers in $\EE^+(Z,N)$ and $\EE^+(\sigma Z,N-1)$. That is, if $w\in\EEE$ for some cover in $\EE^+(Z,N)$, then $[\sigma w]^+\cap \sigma Z\neq\emptyset$.  Likewise, if $w\in\EEE'$ for some cover in $\EE(\sigma Z,N-1)$, then $[x_0w]^+\cap Z\neq\emptyset$.

	Using equation \eqref{eqn:Phi-plus-cont} in the definition of $\mmu_x(Z)$ for $N> N_0$, we get
	\begin{equation}\label{eqn:shift-mxu-bound}
		\begin{split}
			\mmu_x(Z)&=\lim_{N\to\infty}\inf\left\{\sum_{w\in\EEE}e^{\Phi_{x}^+(w)}: \EEE\in\EE^+(Z, N)\right\}\\
			&=\lim_{N\to\infty}\inf\left\{\sum_{w\in\EEE}e^{c\pm\eps+\Phi_{\sigma x}^+(\sigma w)}:\EEE\in\EE^+(Z,N)\right\} \\
			&=e^{c\pm\eps}\lim_{N\to\infty}\inf\left\{\sum_{u\in\EEE'}e^{\Phi_{\sigma x}^+(u)}: \EEE'\in\EE^+(\sigma Z, N-1)\right\}\\
			& = e^{c\pm\eps}\mmu_{\sigma x}(\sigma Z)
		\end{split}
	\end{equation}
	where $u=\sigma w$.  Additionally, because $\ph(y) =c\pm\eps$ on $Z$, we have that
	\begin{equation}\label{eqn:shift-mxu-int}
		\int_{Z}e^{-\ph(y)}\,d\mmu_x(y)=\int_{Z}e^{-c\pm\eps}\,d\mmu_x(y)=e^{-c\pm\eps}\mmu_{x}(Z).
	\end{equation}
		Combining the result from equation \eqref{eqn:shift-mxu-bound} and \eqref{eqn:shift-mxu-int} we get 
	\begin{equation*}
		\int_{Z} e^{-\ph(y)} \,d\mmu_x(y) 
		=e^{-c\pm\eps} \mmu_x(Z)
		=e^{-c\pm\eps} e^{c\pm\eps} \mmu_{\sigma x}(\sigma Z)
		=e^{\pm2\eps} \mmu_{\sigma x}(\sigma Z),
	\end{equation*}
which completes the proof of the lemma.
\end{proof}

To prove Theorem \ref{thm:dynamics}, we consider a set $Z\subset \Wul(x)$ satisfying $\sigma Z\subset\Wul(\sigma x)$.  For $\eps>0$, let $Z_n=Z\cap \ph^{-1}([n\eps/2,(n+1)\eps/2))$ and note that the sets $Z_n$ are measurable, mutually disjoint, and $Z=\bigcup_{n\in\ZZ}Z_n$.   Moreover, 
since $\ph$ is uniformly continuous, there is $N\in\NN$ such that $|\ph(y) - \ph(z)|<\eps/2$ whenever $y_{[-N,N]} = z_{[-N,N]}$, and in particular given any $w\in \LLL$ with $|w|\geq N$ and $[w]^+ \cap Z_n \neq\emptyset$, we can fix $y\in [w]^+\cap Z_n$ and observe that any $z\in [w]^+ \cap \Wul(x)$ has
\[
|\ph(z) - n\eps/2| \leq |\ph(z) - \ph(y)| + |\ph(y) - n\eps/2|
\leq \eps/2 + \eps/2 = \eps.
\]
Thus each $Z_n$ satisfies the conditions of Lemma \ref{lem:shift-mxu-scales}, so
\begin{equation*}
	\mmu_{\sigma x}(\sigma Z)=\sum_{n\in\ZZ}\mmu_{\sigma x}(\sigma Z_n) = e^{\pm2\eps}\sum_{n\in\ZZ}\int_{Z_n}e^{-\ph(y)}\,d\mmu_x(y)
	 = e^{\pm2\eps}\int_Ze^{-\ph(y)}\,d\mmu_x(y).
\end{equation*}
Letting $\eps\to0$ proves \eqref{eqn:shift-mxu-scales}. The proof of \eqref{eqn:shift-mxs-scales} is analogous.

\section{Uniformly continuous holonomies; proof of Theorem \ref{thm:shift-cts}}\label{sec:shift-cts-pf}

In this section we prove the results in \S\ref{sec:holonomies} about uniformly continuous holonomies, starting with Propositions \ref{prop:open-hol} and \ref{prop:zero-hol}, then proceeding to Theorem \ref{thm:shift-cts}.

\subsection{Examples with uniformly continuous holonomies}

Propositions \ref{prop:open-hol} and \ref{prop:zero-hol} both rely on the following fact.

\begin{lemma}\label{lem:same-E}
Given any rectangle $R$, any $y,z\in R$, and any $Y \subset \Wu_R(y)$, we have $\EE^+(Y,N) = \EE^+(\pi^s_{y,z}(Y),N)$ for all $N\in\NN$.
\end{lemma}
\begin{proof}
Given any $N\in \NN$ and $\EEE\in \EE^+(Y,N)$, for every $x\in \pi^s_{y,z}(Y)$ we have $x' = (\pi^s_{y,z})^{-1}(x) \in Y$ and thus $x' \in [w]^+$ for some $w\in \EEE$, but since $x' \in \Wsl(x)$ this implies that $x \in [w]^+$ as well, so $\EEE \in \EE^+(\pi^s_{y,z}(Y),N)$. This proves one inclusion, and the other follows by symmetry.
\end{proof}

With this lemma in hand, Proposition \ref{prop:zero-hol} (the case $\ph\equiv 0$) is easy to prove.

\begin{proof}[Proof of Proposition \ref{prop:zero-hol}]
By Lemma \ref{lem:same-E}, we have $\EE^+(Y,N) = \EE^+(\pi^s_{y,z}(Y),N)$. For every $\EEE \in \EE^+(Y,N)$ and every $w\in \EEE$, we have $\Phi^+_y(w) = \Phi^+_z(w) = -|w| P$ since $\ph\equiv 0$, and thus
$\sum_{w\in \EEE} e^{\Phi^+_y(w)} = \sum_{w\in \EEE} e^{\Phi^+_z(w)}$. Taking an infimum over all such $\EEE$ gives $\mmu_y(Y) = \mmu_z(\pi^s_{y,z}(Y))$.
\end{proof}

Now we prove uniform continuity of holonomies for open rectangles when $\ph$ has the Walters property.

\begin{proof}[Proof of Proposition \ref{prop:open-hol}]
By the Walters property, given $\eps>0$, there is $k\in \NN$ such that for all $n\in \NN$ and $y,z\in R$ with $y_i = z_i$ for all $i\geq -k$, we have $|S_n\ph(y) - S_n\ph(z)| \leq \eps$ for all $n\in\NN$. We claim that $\delta = 2^{-k}$ satisfies the required conclusion in Definition \ref{def:cts-hol}; that is, that for any open rectangle $R$ and any $y,z\in R$ with $d(y,z) < \delta$, we have $\mmu_y(Y) = e^{\pm \eps} \mmu_z(\pi_{y,z}^s(Y))$ for all measurable $Y\subset \Wu_R(y)$.

It suffices to show this conclusion in the case when $R = [u]^- \cap [v]^+$ for some $u,v\in \LLL$. Indeed, \emph{any} open rectangle $R$ can be written as a disjoint union $R = \bigsqcup_{i=1}^\infty R_i$, where each $R_i$ is of the form $[u]^- \cap [v]^+ \cap R = [u]^- \cap [v]^+$, and then if we have the result for each $R_i$, we can deduce that if $d(y,z) < \delta$, then there are $y_i \in \Wu_R(y) \cap R_i$ and $z_i \in \Wu_R(z) \cap R_i$ such that $d(y_i,z_i) < \delta$, so
\[
\mmu_y(Y) = \sum_{i=1}^\infty \mmu_{y_i}(Y\cap R_i) = \sum_{i=1}^\infty e^{\pm \eps} \mmu_{z_i}(\pi_{y,z}^s(Y) \cap R_i) = e^{\pm \eps} \mmu_z(\pi_{y,z}^s(Y)).
\]
To prove the conclusion when $R = [u]^- \cap [v]^+$, we observe that given any $w\in \LLL$ with $|w| \geq |v|$ and $[w]^+ \cap R \neq\emptyset$, we in fact have $[u]^- \cap [w]^+ \subset R$. Given any $x\in \Wu_R(y) \cap [w]^+$, we have $x' := \pi^s_{y,z}(x) \in \Wu_R(z) \cap [w]^+$ and so $x_i = x_i'$ for all $i\geq -k$. By the Walters property, this implies that
\[
|S_n\ph(x) - S_n\ph(x')| \leq \eps \text{ for all } n\in \NN.
\]
Taking a supremum over all $x\in \Wu_R(y) \cap [w]^+$ gives $\Phi_x^+(w) \leq \Phi_{x'}^+(w) + \eps$, and by the symmetrical result with $x,x'$ reversed we get
\begin{equation}\label{eqn:Phixx'}
\Phi_x^+(w) = \Phi_{x'}^+(w) \pm \eps.
\end{equation}
(Note that this step is where we use the openness of $R$ in an essential way; the argument would fail if we did not have $[u]^- \cap [w]^+ \subset R$, since then the holonomy map would only be defined on $\Wul(x) \cap [w]^+ \cap R$, which might not be all of $\Wul(x) \cap [w]^+$.)

Fix $Y\subset \Wu_R(y)$ and let $N \geq |w|$. Then any cover $\EEE \in \EE^+(Y,N)$ is also a cover of $\pi^s_{y,z}(Y) \subset \Wu_R(z)$, and vice versa, by the product structure, and \eqref{eqn:Phixx'} gives
	\begin{align*}
		\sum_{w\in\EEE}e^{\Phi^+_y(w)}=\sum_{w\in\EEE}e^{\Phi^+_z(w)\pm\eps}.
	\end{align*}
	Taking an infimum over covers in $\EE^+(Y,N)$ (and hence also over $\EE^+(\pi^s_{y,z}(Y),N)$) and then sending $N\to\infty$ gives $\mmu_y(Y)=e^{\pm\eps}\mmu_z(\pi^s_{y,z}(Y))$.  This proves the conclusion when $R = [u]^- \cap [v]^+$, and by the discussion at the start of the proof, taking countable unions of such rectangles gives the result for arbitrary open rectangles.
\end{proof}

\subsection{Proof of Theorem \ref{thm:shift-cts}}

Now we prove Theorem \ref{thm:shift-cts}, that uniform continuity of holonomies is equivalent to uniform convergence and boundedness of $\Delta^s$ together with the formula \eqref{eqn:shift-RN} for the Radon--Nikodym derivative of the holonomy map.

Before proving the equivalence, we observe an important consequence of the positivity condition in Definition \ref{def:pos-meas}; given any $R\in \RRR$, any $x\in R$, and any $w\in \LLL$ such that $[w]^+ \cap \Wu_R(x) \neq \emptyset$, we fix $y\in [w]^+ \cap \Wu_R(x)$ and get $\sigma^n(y)$
$\in \sigma^n(R\cap [w]^+)$, which is itself a rectangle in $\RRR$ by $\LLL$-invariance, and thus
\[
\mmu_{\sigma^n(y)} (\sigma^n(R\cap [w]^+)) > 0
\]
by positivity. Applying \eqref{eqn:mxu-RN-n}, we see that $\mmu_x([w]^+) = \mmu_y([w]^+) > 0$, and conclude that $\mmu_x$ is fully supported in the following sense.

\begin{lemma}\label{lem:full-support}
If $\RRR$ is an $\LLL$-invariant family of rectangles on which the family of leaf measures $\mmu_x$ is positive, then for every $R\in \RRR$ and $x\in R$, the measure $\mmu_x$ gives positive weight to every (relatively) open set in $\Wu_R(x)$.
\end{lemma}

Now we prove the equivalence of the two conditions in Theorem \ref{thm:shift-cts}.
Start by assuming that $\Delta^s$ converges uniformly and that \eqref{eqn:shift-RN} is the Radon-Nikodym derivative.\footnote{In fact, uniform convergence of $\Delta^s$ implies boundedness, see Lemma \ref{lem:Delta-bdd}.} We must prove uniform continuity of holonomies.

Observe that each partial sum $\sum_{n=0}^N (\ph(\sigma^n x') - \ph(\sigma^n x))$ is uniformly continuous, and thus by uniform convergence, $\Delta^s$ is uniformly continuous as well. Since $\Delta^s(x,x) = 0$, this implies that given $\delta>0$, there exists $N\in \NN$ such that given any $x,x'\in R\in\RRR$ satisfying $x_i=x'_i$ for all $i \geq -N$, we have $|\Delta^s(x',x)| < \delta$.

Now given any $y,z\in R$ with $d(y,z) < 2^{-N}$, we have $y_i = z_i$ for all $|i|\leq N$, and thus given any $x\in \Wu_R(y)$ and $x' = \pi^s_{y,z}(x)$, we have $x_i = x'_i$ for all $i\geq -N$. Using \eqref{eqn:shift-RN}, it follows that for every $Y\subset \Wu_R(y)$, we have
\[
\mmu_z(\pi_{y,z}^s(Y))=\int_{Y}e^{\Delta^s([z,x],x)}\,d\mmu_y(x)
=\int_Ye^{\pm\delta}\,d\mmu_y(x) =e^{\pm\delta}\,d\mmu_y(Y).
\]

Conversely, suppose that the leaf measures have uniformly continuous holonomies on $\RRR$.  Let $y\in R$ and fix $\delta>0$.  There exists an $N_0\in \NN$ such that if $y'_i=y_i$ for all $i\geq -N_0$, then
\begin{equation}\label{eqn:pmdelta}
\mmu_y(Y)=e^{\pm\delta}\mmu_{y'}(\pi_{y,y'}^s(Y)).
\end{equation}
	Choose $Y\subset\Wu_R(y)$ such that $\sigma^{N}(Y)\subset \Wul(\sigma^{N}y)$.  For simplicity, consider $z\in\Ws_R(y)$, and note that given any $N\geq N_0$, we have
	$(\sigma^{N}z)_{i}=(\sigma^{N}y)_{i}$ for all $i\geq -N$, and hence for all $i\geq -N_0$.   Thus, \eqref{eqn:pmdelta} gives
\begin{equation}\label{eqn:pmdelta-again}
\mmu_{\sigma^{N}y}(\sigma^{N}Y)=e^{\pm\delta}\mmu_{\sigma^{N}z}(\pi^s_{y,z}(\sigma^{N}Y)).
\end{equation}
Using \eqref{eqn:shift-mxu-scales} twice and  \eqref{eqn:pmdelta-again} once, we get
	\begin{align*}
		\mmu_z(\pi^s_{y,z}(Y))&=\int_{\sigma^{N}(\pi^s_{y,z}(Y))}e^{S_N\ph(\sigma^{-N}x)}\,d\mmu_{\sigma^{N}z}(x)\\
		&=e^{\pm\delta}\int_{\sigma^N(Y)}e^{S_N\ph(\sigma^{-N}[\sigma^Nz,x])}\,d\mmu_{\sigma^{N}y}(x)\\
		&=e^{\pm\delta}\int_{Y}e^{S_N\ph(\sigma^{-N}[\sigma^{N}z,\sigma^Nx])-S_N\ph(x)}\,d\mmu_{y}(x)\\
		&=e^{\pm\delta}\int_{Y}e^{S_N\ph([x,z])-S_N\ph(x)}\,d\mmu_{y}(x).
	\end{align*}
The last line follows because for any $x\in \Wul(y)\cap Y$, we have that $x_{[0,N]}=y_{[0,N]}=z_{[0,N]}$ as $\sigma^{N}(Y)\subset\Wul(\sigma^{-N}y)$ and $z\in\Ws_R(y)$.  Thus $\sigma^N[\sigma^{-N}x,\sigma^{-N}z]=[x,z]$.

Since every $Y\subset \Wu_R(y)$ can be decomposed as a disjoint union of sets that lie in a cylinder of length $N$, we conclude that with $N_0 = N_0(\delta)$ as above, if we write $\Delta_N(x',x) := S_N\ph(x') - S_N\ph(x)$ then we have
\begin{equation}\label{eqn:muY}
\mmu_z(\pi^s_{y,z}(Y)) = e^{\pm \delta} \int_{Y}e^{\Delta_N([x,z],x)}\,d\mmu_{y}(x)
\end{equation}
for every $y,z,x$, every measurable $Y\subset \Wu_R(y)$, and every $N\geq N_0$. Thus for every $N,N'\geq N_0$ we have
\begin{equation}\label{eqn:int-close}
\int_Y e^{\Delta_N([x,z],x)} \,d\mmu_y(x) = e^{\pm 2\delta} \int_Y e^{\Delta_{N'}([x,z],x)} \,d\mmu_y(x).
\end{equation}
Since $\mmu_y$ gives positive weight to every relatively open set in $\Wu_R(y)$ by Lemma \ref{lem:full-support}, a standard argument shows that $\Delta_N([x,z],x) = \Delta_{N'}([x,z],x) \pm 2\delta$ for all $x\in \Wu_R(y)$; indeed, if this inequality fails at any point, then by continuity there is a relatively open set $Y$ on which it fails uniformly, and since $\mmu_y(Y)>0$, this would violate the integral estimates in \eqref{eqn:int-close}.

We have proved that $\Delta_N$ is uniformly Cauchy: for every $\delta>0$ there is $N_0\in \NN$ such that for every $R\in \RRR$, every $y,z\in R$, every $x\in \Wu_R(y)$, and every $N,N'\geq N_0$, we have
\[
|\Delta_N([x,z],x) - \Delta_{N'}([x,z],x)| \leq 2\delta.
\]
In particular, given any $x,x'\in R$ with $x' \in \Ws_R(x)$, we can choose $y=x$ and $z=x'$ to deduce that
\[
|\Delta_N(x',x) - \Delta_{N'}(x',x)| \leq 2\delta.
\]
This proves that $\Delta_N$ converges uniformly to $\Delta^s$ on $\{(x,x') \in R^2 : R\in \RRR, x' \in \Ws_R(x)\}$, and then \eqref{eqn:muY} gives
\[
\mmu_z(\pi^s_{y,z}(Y)) = e^{\pm \delta} \int_{Y}e^{\Delta^s([x,z],x)}\,d\mmu_{y}(x).
\]
Since $\delta>0$ is arbitrary, \eqref{eqn:shift-RN} follows.

Finally, we observe that uniform convergence of $\Delta_N$ implies uniform boundedness, by the following lemma.

\begin{lemma}\label{lem:Delta-bdd}
If the sums in the definitions of $\Delta^{s,u}$ converge uniformly, then there is $C>0$ such that $|\Delta^s(x,x')| \leq C$ for all $(x,x')\in R^2$, $R\in \RRR$, $x'\in W_R^s(x)$, and similarly for $\Delta^u$.
\end{lemma}
\begin{proof}
By uniform convergence, there exists $N\in\NN$ such that 
\[
\Big|\sum_{k=N}^\infty\ph(\sigma^{k}x)-\ph(\sigma^{k}x')\Big| \leq 1
\]
for all $x,x'$ as in the statement. Thus
\[
|\Delta^s(x,x')|\leq \sum_{k=0}^{N-1}|\ph(\sigma^{k}x)-\ph(\sigma^{k}x')|+\Big|\sum_{k=N}^\infty \ph(\sigma^{k}x)-\ph(\sigma^{k}x')\Big|
\leq 2N\|\ph\| + 1.
\]
The proof for $\Delta^u$ is analogous.
\end{proof}

\section{Equilibrium measures using product structure}\label{sec:es-two-sided}

In this section we prove Theorems \ref{thm:q-indep}, \ref{thm:return-map}, and \ref{thm:push-product} from \S\ref{sec:product-measure}. Throughout this section, we will assume that $\RRR$ is an $\LLL$-invariant family of rectangles on which the families of leaf measures $\mmu_x$ and $\ms_x$ both have uniformly continuous holonomies and are positive in the sense of Definition \ref{def:pos-meas}.

\subsection{Proof of Theorem \ref{thm:q-indep}}

By Remark \ref{rmk:leafs-related}, the assumption of positivity together with the holonomy results from Theorem \ref{thm:shift-cts} guarantees that $\mmu_q(R)>0$ and $\ms_q(R)>0$ for every $q\in R$. This is enough to imply that $m_{R,q}$ from \eqref{eqn:mxm} is nonzero for every $q\in R$, and since $\Delta^{u,s}$ are uniformly bounded by Theorem \ref{thm:shift-cts}, this in turn implies that $\lambda_R$ is nonzero. It remains to show that $\lambda_R$ is independent of our choice of $q\in R$.

To this end, let $q,q'\in R$. Recalling that
\[
\Delta^u(x',x) := \sum_{n=0}^\infty \ph(\sigma^{-n} x') - \ph(\sigma^{-n} x),
\]
we see that $\Delta^u$ has the following additivity property: given any $y\in \Ws_R(q)$ and $z\in\Wu_R(q)$, we have
\[
\Delta^u([y,z],y)= \Delta^u([y,q'],y)
+\Delta^u([y,z],[y,q']).
\]
	Using this and (the $\ms$-version of) Theorem \ref{thm:shift-cts} we have
	\begin{align*}
		\int_{\Ws_R(q)}&e^{\Delta^u([y,z],y)}e^{\Delta^s([y,z],z)}\one_{Z}([y,z])\,d\ms_q(y)\\
		&=\int_{\Ws_R(q)}e^{\Delta^u([y,z],[y,q'])+\Delta^u([y,q'],y)}e^{\Delta^s([y,z],z)}\one_{Z}([y,z])\,d\ms_q(y)\\
		&=\int_{\Ws_R(q)}e^{\Delta^u([y,z],[y,q'])}e^{\Delta^s([y,z],z)}\one_{Z}([y,z])\,d((\pi^u_{q',q})_*\ms_{q'})(y).
	\end{align*}
Given $y\in \Ws_R(q)$, let $y' = (\pi_{q',q}^u)^{-1}(y)$; then the above equation gives

\begin{multline}\label{eqn:stable-leaf-prod}
\int_{\Ws_R(q)}e^{\Delta^u([y,z],y)}e^{\Delta^s([y,z],z)}\one_{Z}([y,z])\,d\ms_q(y)\\
=\int_{\Ws_R(q')}e^{\Delta^u([y',z],y')}e^{\Delta^s([y',z],z)}\one_{Z}([y',z])\,d\ms_{q'}(y').
\end{multline}

Using \eqref{eqn:stable-leaf-prod} and Fubini's theorem we get
\begin{align*}
\lambda_R(Z) &= \int_{\Wu_R(q)} \int_{\Ws_R(q)} e^{\Delta^u([y,z],y)}e^{\Delta^s([y,z],z)}\one_{Z}([y,z])\,d\ms_q(y) \,d\mmu_q(z) \\
&=\int_{\Ws_R(q')} \int_{\Wu_R(q)} e^{\Delta^u([y',z],y')}e^{\Delta^s([y',z],z)}\one_{Z}([y',z])\,d\mmu_q(z)\,d\ms_{q'}(y').
\end{align*}
The argument leading to \eqref{eqn:stable-leaf-prod} lets us replace $q$ with $q'$ in the integral over $\Ws$, provided we also replaced $y$ with $y' = (\pi_{q',q}^s)^{-1}(y)$. A completely analogous argument for the integral over $\Wu$ lets us deduce that
\[
\lambda_R(Z) = \int_{\Ws_R(q')} \int_{\Wu_R(q')} e^{\Delta^u([y',z'],y')}e^{\Delta^s([y',z'],z')}\one_{Z}([y',z'])\,d\mmu_{q'}(z')\,d\ms_{q'}(y'),
\]
where $z'$ and $z$ are related by $z=\pi^u_{q',q}(z')$. We see that using $q'$ instead of $q$ in \eqref{eqn:mxm} and \eqref{eqn:def-lambda-R} would lead to exactly this formula for $\lambda_R(Z)$, and thus we have proved that $\lambda_R$ is independent of the choice of $q\in R$.

\subsection{Proof of Theorem \ref{thm:return-map}}

Now we assume that $R$ is the disjoint union of finitely many rectangles $R_1,\dots, R_\ell\in \RRR$, and let $\lambda_R = \sum_{i=1}^\ell \lambda_{R_i}$ as in \eqref{eqn:def-lambda}.

Let $R^{(1)} = \{x\in R : \tau(x) < \infty\}$ be the set on which $T$ is defined; observe that $\lambda_R(R\setminus R^{(1)}) = 0$ by hypothesis. Define $R^{(n)}$ as the set of points in $R^{(1)}$ whose forward trajectory stays in $R^{(1)}$ for $n-1$ steps, that is, by $R^{(n)} = R^{(1)} \cap T^{-1}(R^{(n-1)})$; again we have $\lambda_R(R\setminus R^{(n)}) = 0$ for all $n$.  Let $R^{(\infty)} = \bigcap_{n=1}^\infty R^{(n)}$, so we have $T\colon R^{(\infty)} \to R^{(\infty)}$, and $\lambda_R(R\setminus R^{(\infty)})=0$. We will show that $\lambda_R$ is $T$-invariant on $R^{(\infty)}$.

Given $n\in \NN$ and $i,j\in \{1,\dots, \ell\}$, let
\[
E_{ij}^n := \tau^{-1}(\{n\}) \cap R_i \cap \sigma^{-n}(R_j).
\]
We will prove the following.

\begin{lemma}\label{lem:Eijn}
Given any $i,j,n$ as above, any $w\in \LLL_n$, and any measurable $Z \subset E_{ij}^n \cap [w]^+$, we have $\lambda_R(T(Z)) = \lambda_R(Z)$.
\end{lemma}

Once the lemma is proved, observe that for any measurable $Z\subset R^{(\infty)}$ we have $Z = \bigcup_{i,j,n,w} Z \cap E_{ij}^n \cap [w]^+$, and $T(Z) = \bigcup_{i,j,n,w} T(Z\cap E_{ij}^n \cap [w]^+)$, and thus
\[
\lambda_R(T(Z)) = \sum_{i,j,n,w} \lambda_R(T(Z\cap E_{ij}^n \cap [w]^+))
= \sum_{i,j,n,w} \lambda_R(Z \cap E_{ij}^n \cap [w]^+) = \lambda_R(Z),
\]
which proves Theorem \ref{thm:return-map}. So it only remains to prove Lemma \ref{lem:Eijn}.

To this end, fix $i,j,n,w$ as above, and let $Z\subset E_{ij}^n \cap [w]^+$ be measurable. Then $Z\subset R_i \cap [w]^+$, and since $\tau(x)=n$ for all $x\in Z$, we have $T(Z) = \sigma^n(Z) \subset R_j$.
	
Since we assumed that the family $\RRR$ of rectangles is $\LLL$-invariant as in Definition \ref{def:L-inv}, the rectangles $\sigma^n(R_i \cap [w]^+)$ and $\sigma^{-n}(R_j \cap [w]^-)$ both lie in $\RRR$, and hence by Theorem \ref{thm:shift-cts} we can use \eqref{eqn:shift-RN} for holonomies between unstable leaf measures, as well as its analogue for the stable leaf measures.

Recalling \eqref{eqn:shift-bracket}, given any $y,z\in E_{ij}^n \cap [w]^+$, we have $\sigma^n y,\sigma^nz \in R_j$, and $[\sigma^n y, \sigma^n z] = \sigma^n [y,z]$.

Fix a reference point $q\in E_{ij}^n \cap [w]^+$. As a consequence of equation \eqref{eqn:shift-mxs-scales}, we have
\begin{equation}\label{eqn:mxs-f-scales}
\begin{split}
\int_{\Ws_{R_j}(\sigma^nq)\cap [wq_n]^-} f(y)&e^{-\sum_{k=0}^{n-1}\ph(\sigma^{-k}(y))}\,d\ms_{\sigma^nq}(y)\\
&=\int_{\Ws_{R_i}(q)} f(\sigma^n(y))\,d\ms_{q}(y)
\end{split}
\end{equation}
where $f$ is integrable.  Motivated by this expression, observe that if $y\in \Ws_{R_j}(\sigma^n q)$ and $z\in\Wu_{R_j}(\sigma^n q)$, then
\begin{equation}\label{eqn:Deltau-pushforward}
\begin{split}
\Delta^u([y,z], y)&+\sum_{k=0}^{n-1}\ph(\sigma^{-k} y)\\
&=\sum_{k=0}^{n-1}\ph(\sigma^{-k}[ y,z])+\Delta^u(\sigma^{-n}[ y,z],\sigma^{-n}y).
\end{split}
\end{equation}
Similarly, equation \eqref{eqn:shift-mxu-scales} implies
\begin{equation}\label{eqn:mxu-f-scales}
\int_{\Wu_{R_j}(\sigma^n q)}f(z)\,d\mmu_{\sigma^n q}(y)=\int_{\Wu_{R_{i}}(q)}f(\sigma^ny)e^{-S_n\ph(y)}\,d\mmu_q(y).
\end{equation}
Moreover,
\begin{equation}\label{eqn:Deltas-pushforward}
\begin{split}
\Delta^s([y,\sigma^n z],\sigma^n z)&-S_n\ph(z)\\
&=\Delta^s(\sigma^{-n}[y,\sigma^n z],z)-\sum_{k=0}^{n-1}\ph(\sigma^{-k}[y,\sigma^n z]).
\end{split}
\end{equation}
Using \eqref{eqn:Deltas-pushforward} followed by \eqref{eqn:mxs-f-scales} on the integral
\[
\int_{\Wu_{R_j}(Tq)}\int_{\Ws_{R_j}(Tq)} e^{\Delta^u([y,z],y)+\Delta^s([y,z],z)}\one_{T(Z_n(w))}([y,z])\,d\ms_{Tq}(y)\,dm^u_{Tq}(z),
\]
the term in the exponent becomes
\begin{equation*}
\Delta^u(\sigma^{-n}[\sigma^n y,z],y)+\sum_{k=0}^{n-1}\ph(\sigma^{-k}[\sigma^n y,z]) + \Delta^s([\sigma^ny,z],z).
\end{equation*}
Then applying \eqref{eqn:mxu-f-scales} and then \eqref{eqn:Deltas-pushforward} the exponent becomes
\begin{multline*}
\Delta^u(\sigma^{-n}([\sigma^n y, \sigma^n z]), y)+\sum_{k=0}^{n-1}\ph(\sigma^{-k}[\sigma^n y,\sigma^n z])\\
+\Delta^s(\sigma^{-n}([\sigma^n y, \sigma^n z]), z)-\sum_{k=0}^{n-1}\ph(\sigma^{-k}[\sigma^n y,\sigma^n z]).
\end{multline*}
Since $\sigma^{-n}([\sigma^n y,\sigma^n z])=[y,z]$ this term simplifies to 
\[
\Delta^u([y, z], y)+\Delta^s([y,z], z).
\]
For convenience, let
\[
\rho(y,z)=e^{\Delta^u([y, z], y)+\Delta^s([y,z], z)}.
\]
Altogether, we have shown that given $Z\subset E_{ij}^n \cap [w]^+$, we have
\begin{align*}
\lambda_R(TZ)
& = \int_{\Wu_{R_j}(T q)}\int_{\Ws_{R_j}(Tq)} \rho(y',z')\one_{TZ}[y', z']\,d\ms_{Tq}(y')\,dm^u_{Tq}(z')\\
&=\int_{\Wu_{R_i}(q)}\int_{\Ws_{R_i}(q)} \rho(y,z)\one_{TZ}[\sigma^n y,\sigma^n z]\,d\ms_q(y)\,dm^u_q(z)\\
& = \int_{\Wu_{R_i}(q)}\int_{\Ws_{R_i}(q)} \rho(y,z)\one_{Z}[y,z]\,d\ms_q(y)\,dm^u_q(z)
= \lambda_R(Z).
\end{align*}
This proves Lemma \ref{lem:Eijn}, and as explained in the paragraph following the lemma, we deduce that $\lambda_R$ is invariant with respect to the return map on $R^{(\infty)}$.

\subsection{Proof of Theorem \ref{thm:push-product}}

Now we assume that in addition to the conditions of the previous sections, we have $\int \tau\,d\lambda_R < \infty$ and $\uQ<\infty$. We must show that the measure $\mu$ in \eqref{eqn:es-mu} is positive, finite, and $\sigma$-invariant, and that $\mu/\mu(X)$ is an equilibrium measure with local product structure.

Positivity of $\mu$ is immediate from positivity of $\lambda_R$. Finiteness of $\mu$ follows from the integrability assumption:
\[
\mu(X) = \sum_{n=0}^\infty \lambda_R(Y_n)
= \sum_{n=0}^\infty \lambda_R( \{y\in R : \tau(y) > n \}) = \int \tau \,d\lambda_R.
\]
Shift-invariance of $\mu$ follows immediately from $T$-invariance of $\lambda_R$ by the usual argument. Finally, local product structure in the sense of Definition \ref{def:lps} follows by considering the rectangles $R_i$ together with the rectangles $\sigma^k(E_{ij}^n \cap [w]^+)$ for $w\in \LLL_n$ and $1\leq k < n$. So it only remains to show that $\mu/\mu(X)$ is an equilibrium measure. For this we need the upper Gibbs bound in Corollary \ref{cor:lambda-Gibbs}, which implies that for every $w\in \LLL$ and $1\leq i\leq \ell$, we have $\lambda_{R_i}([w]^+) \leq K e^{\Phi(w)}$ (recall we are assuming $P=0$), and thus
\begin{equation}\label{eqn:lambda-w-upper-bound}
\lambda_R([w]^+)\leq \ell K e^{\Phi(w)}.
\end{equation}
In order to prove that $\mu$ is an equilibrium measure, we will need one more preliminary result. Let $V_j(\ph)=\sup\{|\ph(y)-\ph(z)|:y_{[0,j)}=z_{[0,j)}\}$.  Recall $\tau\colon R^{(\infty)} \to\NN$ is the first return time to $R$ and let $\tau_n(x)$ denote the $n^{\mathrm{th}}$-return time of $x$ to $R$, defined as
\begin{equation}
	\tau_n(x)=\sum_{j=0}^{n-1}\tau(T^jx).
\end{equation}
We will  require the following lemma.  Note that the dependence of $x$ is in the upper index of the sum and not in the terms of the summand itself as $V_j(\ph)$ is defined globally.

\begin{lemma}\label{lem:Phi-Sn-lim}
	If $\tau$ is integrable with respect to $\lambda_R$, then
	\begin{equation}
		\lim_{n\to\infty}\frac{1}{n}\sum_{j=0}^{\tau_n(x)-1}V_j(\ph)=0
	\end{equation}
	for $\lambda_R$-a.e. $x\in R$.
\end{lemma}
\begin{proof}
	For all $\eps>0$, there exists an $N_1$ such that for all $n\geq N_1$ we get $V_n(\ph)<\eps$ by uniform continuity.
	
	By the Birkhoff ergodic theorem, for $\lambda_R$-a.e. $x$, 
	\[\lim_{n\to\infty}\frac{\tau_n(x)}{n}=\tau_\infty(x)\in L^1(\lambda_R),\]
	so there exists an $N_2$ such that for $n\geq N_2$, we have that $0\leq \tau_n(x)/n<\tau_\infty(x)+\eps$.  Additionally, for $\lambda_R$-a.e. $x$, we have that $\tau_\infty(x)<\infty$, so if $n\geq N = \max\{N_1,N_2\}$,
	\begin{align*}
		\frac{1}{n}\sum_{j=0}^{\tau_n(x)-1}V_j(\ph)&=\frac{1}{n}\sum_{j=0}^{N-1}V_j(\ph)+\frac{1}{n}\sum_{j=N}^{\tau_n(x)-1}V_j(\ph)
		 \leq \frac{1}{n}\sum_{j=0}^{N-1} V_j(\ph)+\frac{1}{n}\sum_{j=N}^{\tau_n(x)-1}\eps\\
		& = \frac{1}{n}\sum_{j=0}^{N-1} V_j(\ph)+\frac{\tau_n(x)-N}{n}\eps
		 < \frac{1}{n}\sum_{j=0}^{N-1} V_j(\ph)+\eps(\tau_\infty(x) +\eps).
	\end{align*}
	Letting $n\to\infty$ we get 
	\[\lim_{n\to\infty}\frac{1}{n}\sum_{j=0}^{\tau_n(x)-1}V_j(\ph)\leq\eps(\tau_\infty(x)+\eps).\]
	Since $\eps$ was arbitrary, we get that for $\lambda_R$-a.e. $x\in R$,
	\[\lim_{n\to\infty}\frac{1}{n}\sum_{j=0}^{\tau_n(x)-1}V_j(\ph)=0.\qedhere
\]
\end{proof}

To show that $\mu$ as defined in \eqref{eqn:es-mu} is an equilibrium measure, we will first show that the measure theoretic pressure for $\lambda_R$ is 0 for the induced system when $\ph$ is normalized to have 0 topological pressure.  This will imply that the measure theoretic pressure for $\mu$ is also 0.  We now present the proof of Theorem \ref{thm:push-product}.

\begin{proof}	
	Let $\tilde\ph(x)=S_{\tau(x)}\ph(x)$.  If $X_n=\bigcup_{j=0}^{n-1}\sigma^j Z_n$, then
	\[\mu(X)\int_{X_n}\ph\,d\mu=\int_{Z_n}S_n\ph(x)\,d\lambda_R=\int_{Z_n}\tilde\ph\,d\lambda_R.\]	
	Hence, $\mu(X)\int_{X}\ph\,d\mu=\int_{R}\tilde\ph\,d\lambda_R.$
	
	Similarly, $\mu(X) h_{\mu}(\sigma)=h_{\lambda_R}(T)$ by applying Kac's formula, so we have that 
	\[\mu(X) \left(h_{\mu}(\sigma)+\int_X\ph\,d\mu\right)=h_{\lambda_R}(T)+\int_R{\tilde{\ph}}\,d\lambda_R.\]
	Since we are assuming that the topological pressure $P(\ph)=0$, it is sufficient to show that  $h_{\lambda_R}(T)+\int_R{\tilde{\ph}}\,d\lambda_R\geq 0$.
	
Now we use the Shannon--McMillan--Breiman theorem \cite[p.\ 265]{Glasner} to deduce that if we write
	\[I_n^{\lambda_R}(x):=-\log(\lambda_R([x_0\dots x_{\tau_n(x)-1}]^+)),\]
	then there exists an $L^1(\lambda_R)$ function $J$ such that
	\begin{equation}
		\lim_{n\to\infty}\frac{1}{n}I_n^{\lambda_R}\to J
	\end{equation}
	for $\lambda_R$-a.e. $x$, and
	\[\int_R J\,d\lambda_R=h_{\lambda_R}(T).\]
	By the Birkhoff ergodic theorem, there exists a function $\tilde\ph_\infty\in L^1(\lambda_R)$ such that 
	\[\lim_{n\to\infty}\frac{1}{n}S_n\tilde\ph\to\tilde\ph_{\infty}\]
	for $\lambda_R$-a.e. $x$, and
	\[\int_R\tilde\ph\,d\lambda_R=\int_R\tilde\ph_{\infty}\,d\lambda_R.\]
	Using \eqref{eqn:lambda-w-upper-bound}, we get
	\begin{equation}
		I_n^{\lambda_R}(x)\geq -\log(\ell K)-\Phi(x_0\dots x_{\tau_n(x)-1}),
	\end{equation}
and rearranging terms, we obtain
	\[I_n^{\lambda_R}(x)+\Phi(x_0\dots x_{\tau_n(x)-1})\geq-\log(\ell K),\]
so that in particular,
	\[\varliminf_{n\to\infty}\frac{1}{n}I_n^{\lambda_R}(x)+\frac{1}{n}\Phi(x_0\dots x_{\tau_n(x)-1})\geq0.\]
	Lastly, we must show that $\frac{1}{n}\Phi(x_0,\dots x_{\tau_n(x)-1})\to\tilde\ph_\infty$ a.e.
	\begin{align*}
		\Phi(x_0,\dots x_{\tau_n(x)-1})-S_{\tau_n(x)}\ph(x)
		& = \sup_{y\in[x_0\dots x_{k-1}]^+}(S_{\tau_n(x)}\ph(y)-S_{\tau_n(x)}\ph(x))\\
		& \leq \sum_{j=0}^{\tau_n(x)-1}V_j(\ph).
	\end{align*}
	By Lemma \ref{lem:Phi-Sn-lim}, for $\lambda_R$-a.e. $x\in R$ we have
	\[\varliminf_{n\to\infty}\frac{1}{n}\left(\Phi(x_0\dots x_{\tau_n(x)-1})-S_{\tau_n(x)}\ph(x)\right)=0,\]
	which means that $\frac{1}{n}\Phi(x_0\dots x_{\tau_n(x)-1})\to\tilde\ph_\infty$.
	We conclude that $h_{\lambda_R}(T)+\int \tilde\ph\,d\lambda_R=J(x)+\tilde\ph_{\infty}\geq 0$, so $\mu$ is an equilibrium measure.
\end{proof}

\appendix
\section{Pressure of factorial collections}\label{app:pressure}

We prove Lemma \ref{lem:factorial} following the argument in \cite[\S4.4]{CP19}.
First observe that the argument in \eqref{eqn:sub-mult} showing that $\Lambda_{m+n}(\LLL) \leq \Lambda_m(\LLL) \Lambda_n(\LLL)$ works equally well when $\LLL$ is replaced by any factorial $\DDD$, and thus
\begin{equation}\label{eqn:gen-submult}
\Lambda_{n+m}(\DDD) \leq \Lambda_n(\DDD)\Lambda_m(\DDD).
\end{equation}
Given this submultiplicativity, Fekete's lemma implies that the following limit exists:
\begin{equation}\label{eqn:gen-P(D)}
P(\DDD) := \lim_{n\to\infty} \frac{1}{n}\log \Lambda_n(\DDD)
= \inf_{n\in\NN} \frac 1n \log \Lambda_n(\DDD).
\end{equation}
This limit is finite whenever $\DDD$ is non-empty; indeed, if $A$ is the alphabet for the shift space, then $e^{-n\|\ph\|} \leq \Lambda_n(\DDD) \leq e^{n\|\ph\|} (\#A)^n$ for all $n$, which gives $-\|\ph\| \leq P(\DDD) \leq \|\ph\| + \log(\#A)$.

To prove Lemma \ref{lem:factorial}, it suffices to prove that $P(X_\DDD,\ph) \geq P(\DDD)$, since the other inequality is immediate from the inclusion $\LLL(X_\DDD) \subset \DDD$.
To this end, fix $\eps>0$ and use \eqref{eqn:gen-P(D)} to get $N_0$ such that for all $n\geq N_0$ we have
\begin{equation}\label{eqn:P-eps}
e^{n P(\DDD)} \leq \Lambda_n(\DDD) \leq e^{n(P(\DDD)+\eps)}.
\end{equation}
Given $j\in \NN$, define
\begin{equation}\label{eqn:Dj}
\DDD^{(j)} := \{w \in \DDD : \text{there exist }u,v\in \DDD_j \text{ such that } uwv \in \DDD \}.
\end{equation}
For every $N,j \in \NN$, we have $\DDD_{N+2j} \subset \DDD_j \DDD_N^{(j)} \DDD_j$; that is, every word $w\in \DDD_{N+2j}$ has the property that $w_{(j,j+N]} \in \DDD^{(j)}$.
The same argument that gave \eqref{eqn:gen-submult} gives
\begin{equation}\label{eqn:N2j}
\Lambda_{N+2j}(\DDD) \leq \Lambda_j(\DDD)^2 \Lambda_N(\DDD^{(j)}).
\end{equation}
Moreover, since $\DDD^{(j)}$ is factorial for each $j$, we can iterate \eqref{eqn:gen-submult} to obtain
\[
\Lambda_{kn}(\DDD^{(j)}) \leq \Lambda_n(\DDD^{(j)})^k
\]
for every $k,n,j\in \NN$. Using \eqref{eqn:P-eps} and \eqref{eqn:N2j} with $N=kn$, $j\geq N_0$, this gives
\begin{equation}
e^{(kn+2j)P(\DDD)} \leq \Lambda_{kn+2j}(\DDD) \leq \Lambda_j(\DDD)^2 \Lambda_{kn}(\DDD^{(j)}) 
\leq e^{2j(P(\DDD)+\eps)} \Lambda_n(\DDD^{(j)})^k.
\end{equation}
Dividing both sides by $e^{2jP(\DDD)}$ gives
\[
e^{knP(\DDD)} \leq e^{2j\eps} \Lambda_n(\DDD^{(j)})^k,
\]
at which point we can take logs and divide by $kn$ to get
\[
P(\DDD) \leq \frac{2j}{kn}\eps + \frac 1n \log \Lambda_n(\DDD^{(j)}).
\]
Fixing $n\in \NN$ and $j\geq N_0$, we can send $k\to\infty$ and obtain
\begin{equation}\label{eqn:PDj}
P(\DDD) \leq \frac 1n \log \Lambda_n(\DDD^{(j)}).
\end{equation}
Observe that the sets $\DDD^{(j)}$ decrease monotonically to $\LLL(X_\DDD)$: in particular, for every $n\in \NN$, we have
\[
\DDD_n = \DDD_n^{(0)} \supset \DDD_n^{(1)} \supset \DDD_n^{(2)} \supset \cdots \supset \LLL_n(X_\DDD)
\quad\text{and}\quad
\LLL_n(X_D) = \bigcap_{j\in \NN} \DDD_n^{(j)}.
\]
Since $\DDD_n$ is finite, there is $j=j(n) \geq N_0$ such that $\DDD^{(j)}_n = \LLL_n(X_\DDD)$, so \eqref{eqn:PDj} gives
\[
P(\DDD) \leq \frac 1n \log \Lambda_n(\DDD^{(j(n))}) = \frac 1n \log \Lambda_n(X_\DDD)
\]
for all $n$. Sending $n\to\infty$ proves Lemma \ref{lem:factorial}.

\section{Uniqueness with non-uniform specification}\label{app:unique}

\subsection{Non-uniform specification for flows}

We will deduce Theorem \ref{thm:CT-PYY} from \cite[Theorem A]{PYY22}. First we need the following definitions from \cite{CT16,PYY22}. Throughout, $(f_s)_{s\in\RR}$ will denote a continuous flow on a compact metric space $(M,d)$.

\begin{definition}\label{def:almost-expansive}
For $\eps >0$ and $x \in M$, the \emph{two-sided infinite Bowen ball} at $x$ is
\[
\Gamma_{\eps}(x) = \{y \in M : d(f_s(x),f_s(y)) \leq \eps \text{ for all } s \in \RR\}.
\]
A flow $f_s$ is \emph{expansive at scale $\eps$} if there exists an $s_0>0$, such that $\Gamma_{\eps}(x) \subset f_{[-s_0,s_0]}(x)$ for every $x \in X$.
\end{definition}

We identify a pair $(x,T) \in M \times [0,\infty)$ with the parametrized curve of the orbit segment starting at $x$ flowing for time $T$ under $f_s$.  Thus $M\times [0,\infty)$ represents the space of finite-length orbit segments.

\begin{definition}\label{def:decomp}
A \emph{decomposition} $(\PPP,\GGG,\SSS)$ for $M \times [0,\infty)$ consists of three collections $\PPP,\GGG,\SSS \subset M \times [0,\infty)$ and three functions $P,G,S \colon M \times [0,\infty) \to [0,\infty)$ such that for every $(x,T) \in M \times [0,\infty)$, the values $P(x,T)$, $G(x,T)$, and $S(x,T)$ satisfy $T = P(x,T)+G(x,T)+S(x,T)$, and 
\[(x,P(x,T)) \in \PPP, \quad (f_{P(x,T)}(x),G(x,T)) \in \GGG, \quad (f_{P(x,T)+G(x,T)}(x),S(x,T)) \in \SSS.\]
\end{definition}

As in Theorem \ref{thm:CT-PYY}, we will require $\GGG$ to have specification and the Bowen property, and $\PPP,\SSS$ to be small from the point of view of pressure. In this flow setting, these definitions take the following form.

\begin{definition}\label{def:spec}
A collection of orbit segments $\GGG \subset M \times [0,\infty)$ has \emph{specification} at scale $\delta$ if there exists $t > 0$ such that for every finite sequence $\{(x_i,T_i)\}_{i=1}^k \subset \GGG$, there exists a point $y$ such that writing $s_j = \sum_{i=1}^{j-1} (T_i + t)$, we have
\[
d_{T_j}(f_{s_{j}}(y),x_j) < \delta \text{ for every } 1 \leq j\leq k.
\]
\end{definition}

Given $(x,T) \in M\times [0,\infty)$ and $\eps>0$, the \emph{(closed) Bowen ball} of order $T$ around $x$ with radius $\eps$ is
$\bar{B}_T(x,\eps) = \{y \in M : d(f_s(x),f_s(y)) \leq \eps \text{ for all }s\in [0,T]\}$.  
The remaining definitions involve a continuous potential $\hat{\ph} \colon M \to \RR$.
Define
\begin{equation*}
\Phi_\eps(x,T) = \sup_{y\in \bar{B}_T(x,\eps)}\int_0^T \hat{\ph}(f_s(y)) \,ds.
\end{equation*}
For $\eps = 0$, we have $\Phi_0(x,T) = \int_0^T \hat{\ph}(f_s(x))\,ds$.

\begin{definition}\label{def:Bowen-property}
$\hat{\ph}$ has the \emph{Bowen property at scale} $\eps>0$ on $\GGG \subset M \times [0,\infty)$ if there exists $K>0$ such that
\[
\sup\{|\Phi_0(x,T)-\Phi_0(y,T)| : (x,T) \in \GGG,\, y \in \bar{B}_T(x,\eps)\} \leq K.
\]
\end{definition}

The prefix and suffix collections $\PPP,\SSS$ are required to be small in the sense of pressure:

\begin{definition}\label{def:2-scale-P}
	For $\CCC \subset M \times [0,\infty)$, let $\CCC_T = \{x \in M : (x,T) \in \CCC\}$, and define 
	\[
	\Lambda(\CCC,\hat{\ph},\delta,\eps, T) = \sup\left\{\sum_{x \in E} e^{\Phi_\eps(x,T)} : E \subset \CCC_T \text{ is } (T,\delta)\text{-separated}\right\}.
	\]
	The \emph{pressure of $\CCC$ at scales $\delta,\eps$} is 
	\[
	P(\CCC,\hat{\ph},\delta,\eps) = \varlimsup_{T\to\infty} \frac{1}{T} \log(\Lambda(\CCC,\hat{\ph},\delta,\eps, T)).
	\]
\end{definition}

Since $T$ takes continuous values, care must be exercised in applying Definition \ref{def:2-scale-P} in a situation where removing short pieces from the beginning and end of an orbit segment in $\CCC$ could produce an orbit segment not in $\CCC$; one could have many orbit segments in $\CCC$ with similar but not identical values of $T$, and Definition \ref{def:2-scale-P} would never see all of them at once. Thus we make one final definition:  let
\[
[\CCC] = \{(x,n) \in M \times \NN : (f_{-a}(x),n+a+b) \in \CCC \text{ for some } a,b \in [0,1)]\}
\]
be the collection of integer-length orbit segments that are obtained by removing pieces of length $\leq 1$ from the ends of elements of $\CCC$.

Now we can state a general result for flows, which is an immediate consequence of \cite[Theorem A, (3.2), and Lemma 6.6]{PYY22}.\footnote{The result in \cite{PYY22} uses a non-uniform expansivity assumption which is more complicated to define. The uniform expansivity assumption is satisfied in our application and implies the condition there. Additionally, \cite{PYY22} requires ``tail (W)-specification," which follows from Definition \ref{def:spec}.} 

\begin{theorem}\label{thm:PYY}
Let $(f_s)_{s\in\RR}$ be a continuous flow on a compact metric space $M$, and $\hat\ph \colon M \to \RR$ a continuous potential function.  Suppose there exist $\eps,\delta>0$ with $\eps\geq 2000\delta$ such that the flow is expansive at scale $\eps$, and it admits a decomposition $(\PPP,\GGG,\SSS)$ of orbit segments with the following properties:
\begin{enumerate}
\item \label{item:PYY-spec}	$\GGG$ has specification at scale $\delta$;
\item 	$\ph$ has the Bowen property at scale $\eps$ on $\GGG$;
\item 	$P([\PPP] \cup [\SSS],\hat\ph,\delta,\eps) < P(\hat\ph)$. 
\end{enumerate}
Then there exists a unique equilibrium measure for the potential $\hat\ph$, and for every $\gamma \in [8\delta, 200\delta]$, we have $\varlimsup_{T\to\infty} \Lambda(M\times [0,\infty), \hat\ph,2\gamma,2\gamma,T) e^{-TP(\hat\ph)} < \infty$.
\end{theorem}

\subsection{A suspension flow}

The strategy we use to obtain Theorem \ref{thm:CT-PYY} from Theorem \ref{thm:PYY} is to compare the shift space $X$ to a suspension flow over $X$.  Let $\hat{X} = X\times \RR / {\sim}$, where we quotient by the equivalence relation $(x,a+n) \sim (\sigma^nx, a)$ for all $x\in X$, $a\in \RR$, and $n\in \ZZ$. Observe that we also have $\hat{X} = X \times [0,1] / {\sim}$. We will routinely write $(x,a)$ as a shorthand for its equivalence class.

Let $f_s \colon \hat{X} \to \hat{X}$ be the suspension flow over $X$ with a constant roof function of 1; that is, the quotient of the vertical flow on $X\times \RR$.
Fix $\theta = \frac 1{2000}$, $\eps \in (\theta^2, \theta)$, and $\delta = \eps\theta$, so the relation between $\eps$ and $\delta$ in Theorem \ref{thm:PYY} is satisfied. Equip the shift space $X$ with the metric
\[
d(x,y) := \theta^{n}\text{, where } n = \min \{|m| : x_m \neq y_m\},
\]
and define a metric $\hat{d}$ on $\hat{X}$ following the procedure in \cite[\S4]{BW72}, so that for all $x,y\in X$ and $a\in [0,1]$, we have
\begin{align}\label{eqn:dxya}
\hat{d}((x,a),(y,a)) &\leq ad(\sigma x, \sigma y) + (1-a) d(x,y), \\
\label{eqn:dxy12}
\hat{d}((x,\tfrac 12), (y,a)) &\geq \min(d(x,y), d(\sigma x, \sigma y)).
\end{align}

\begin{lemma}\label{lem:hat-d-leq}
If $\hat{d}((x,\frac 12),(y,a)) \leq \eps$, then $x_0 = y_0$ and $x_1 = y_1$.
\end{lemma}
\begin{proof}
If $\hat{d}((x,\frac 12), (y,a)) \leq \eps < \theta$, then \eqref{eqn:dxy12} implies that either $d(x,y) < \theta$ or $d(\sigma x, \sigma y) < \theta$. In the first case, we have $x_{-1} x_0 x_1 = y_{-1} y_0 y_1$; in the second case, we have $x_0 x_1 x_2 = y_0 y_1 y_2$.
\end{proof}

A direct consequence of Lemma \ref{lem:hat-d-leq} is that the suspension flow is expansive at scale $\eps$. Indeed, if $(y,b) \in \Gamma_\eps((x,a))$, then without loss of generality we can assume that $(y,b)$ is the element if its equivalence class with  $|b-a|\leq \frac 12$, and applying $f_{n + \frac 12 - a}$ to both $(y,b)$ and $(x,a)$ for every $n\in \ZZ$, Lemma \ref{lem:hat-d-leq} gives $x = y$.

\begin{lemma}\label{lem:hat-d-Bn}
If $x,y\in X$ and $n\in \NN$ are such that $x_k = y_k$ for all $-2 \leq k \leq n+2$, then $(y,0) \in B_n((x,0),\xi)$ for each $\xi>\theta^3$.
\end{lemma}
\begin{proof}
If $\hat{d}((x,a),(y,a)) \geq \xi > \theta^3$, then either $d(x,y) > \theta^3$ or $d(\sigma x, \sigma y) > \theta^3$, which in turn implies that $x_k \neq y_k$ for some $k \in \{-2,-1,0,1,2,3\}$. Applying this to $(\sigma^j x, a)$ and $(\sigma^j y, a)$ for all $0\leq j < n$ and $a\in [0,1]$ gives the result.
\end{proof}

Given a continuous $\ph \colon X\to \RR$, we define $\hat{\ph} \colon \hat{X} \to \RR$ by
\[\hat{\ph}(x,a) = a\ph(\sigma x) + (1-a)\ph(x).\]
This choice of $\hat\ph$ leads to the following property.

\begin{lemma}\label{lem:hat-ph}
Given any $x\in X$ and $n\in \NN$, we have
\[
|\Phi_0((x,0),n) - S_n\ph(x) | \leq \|\ph\|.
\]
\end{lemma}
\begin{proof}
Observe that $\int_0^1 \hat\ph(x,a) \,da = \ph(x) + \frac 12 (\ph(\sigma x) - \ph(x))$, so
\begin{equation}\label{eqn:int-hat-ph}
\begin{split}
\int_0^n \hat\ph(f_s(x,0)) \,ds & =  \sum_{k=0}^{n-1} \ph(\sigma^k x) + \frac 12 (\ph(\sigma^{k+1} x) - \ph(\sigma^k x)) \\
& = S_n\ph(x) + \frac 12 (\ph(\sigma^n x) - \ph(x)),
\end{split}
\end{equation}
which proves the lemma.
\end{proof}

There is a natural bijection between shift-invariant probability measures $\mu$ on $X$ and flow-invariant probability measures $\hat\mu$ on $\hat X$, obtained by taking the product of $\mu$ with Lebesgue measure on $[0,1]$. This bijection preserves entropy and has the property that $\int \ph\,d\mu = \int \hat\ph\,d\hat\mu$. From the variational principle, we conclude that $P(X,\ph) = P(\hat{X},\hat\ph)$, and that $\mu$ is an equilibrium measure for $(X,\sigma,\ph)$ if and only if $\hat\mu$ is an equilibrium measure for $(\hat{X},(f_s),\hat\ph)$. In particular, the uniqueness conclusion in Theorem \ref{thm:PYY} implies the uniqueness conclusion of Theorem \ref{thm:CT-PYY}. The fact that the upper bound in Theorem \ref{thm:PYY} implies $\uQ<\infty$ in Theorem \ref{thm:CT-PYY} is a consequence of the following.

\begin{lemma}\label{lem:counting-susp}
Given $\gamma \in [8\delta,200\delta]$, we have
\[
\Lambda_n(X) \leq e^{10\|\ph\|} \Lambda(M\times [0,\infty), \hat\ph, 2\gamma, 2\gamma, n-1).
\]
\end{lemma}
\begin{proof}
For each $w\in \LLL_n$, choose $x_w \in [w]^+$, and consider the set $E = \{(x_w, 0) : w\in \LLL_n\}$. Observe that if $x,y\in X$ are such that $(y,0) \in B_{n-1}((x,0),2\gamma)$, then we have $\hat{d}((\sigma^k x,\frac 12),(\sigma^k y,\frac 12)) < 2\gamma < \eps$ for all $0\leq k<n-1$, and by Lemma \ref{lem:hat-d-leq}, this gives $x_{[0,n)} = y_{[0,n)}$. Thus $E$ is $(n-1,2\gamma)$-separated for the flow. Moreover, for every $w\in \LLL_n$ and $y\in [w]^+$, we have $(y,2) \in B_{n-5}((x,2),\delta)$ by Lemma \ref{lem:hat-d-Bn}, so
\begin{align*}
S_n\ph(y) &\leq 5\|\ph\| + S_{n-5}\ph(\sigma^2 y) \leq 6\|\ph\| + \Phi_{\delta}((x,2),n-5) \\
&\leq 10\|\ph\| + \Phi_{2\gamma}((x,0),n-1),
\end{align*}
where the second inequality uses Lemma \ref{lem:hat-ph}.
Taking a supremum over $y\in [w]^+$ gives $\Phi(w) \leq 10\|\ph\| + \Phi_{2\gamma}((x,0),n-1)$. Summing over $w$ proves the lemma.
\end{proof}

To deduce Theorem \ref{thm:CT-PYY} from \ref{thm:PYY}, it remains to show that if $X$ and $\ph$ satisfy the conditions of Theorem \ref{thm:CT-PYY}, then $\hat{X}$ and $\hat{\ph}$ satisfy the conditions of Theorem \ref{thm:PYY}. Note that expansivity of the suspension flow was proved above via Lemma \ref{lem:hat-d-leq}. We address the remaining conditions in the next section.

\subsection{Non-uniform specification for the suspension flow}

\subsubsection{A decomposition}

The hypothesis of Theorem \ref{thm:CT-PYY} provides a decomposition $\CCC^p$, $\GGG$, $\CCC^s$ of the language $\LLL$ of the shift space $X$; Figure \ref{fig:orbit-decomp} illustrates a corresponding decomposition of orbit segments for the suspension flow.  

\begin{figure}[htbp]
\includegraphics[width=.4\textwidth]{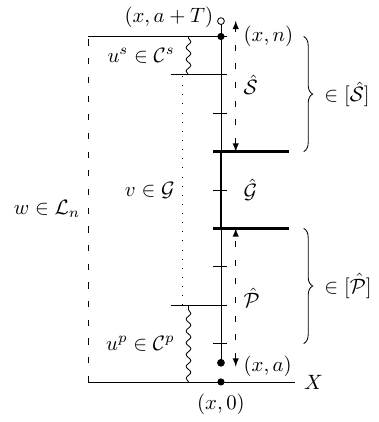}
\caption{Decomposing $((x,a),T)$.  Tick marks indicate integers.}
\label{fig:orbit-decomp}
\end{figure}

Recall that $((x,a),T)$ represents the orbit segment of length $T$ starting at $(x,a) \in \hat{X}$.  The collections of orbit segments illustrated in Figure \ref{fig:orbit-decomp} are
\begin{align*}
\hat{\PPP} &= \{((x,a),k+r) \colon k\in \NN,\ a\in [0,1),\ x_0 \cdots x_{k-1} \in \CCC^p, \text{ and } r \in [0,2] \}, \\
\hat{\GGG} &= \{((x,0),n) \colon n\in \NN \text{ and } x_{-2} x_{-1} x_0 \cdots x_n x_{n+1} x_{n+2} \in \GGG \}, \\
\hat{\SSS} &= \{((x,0),k+r) \colon k\in \NN,\ x_j \cdots x_{k-1} \in \CCC^s \text{ for some } j\in \{0,1,2\}, \text{ and } r \in [0,1) \}.
\end{align*}
To decompose $((x,a),T) \in \hat{X} \times [0,\infty)$ with $a\in [0,1)$, let $n = \lfloor a + T \rfloor$ and consider the word $w = x_0 x_1 \cdots x_{n-1} \in \LLL$. Let $w = u^p v u^s$, where $u^p \in \CCC^p$, $v\in \GGG$, and $u^s \in \CCC^s$.

Informally, the decomposition of $((x,a),T)$ is given by removing two symbols from either end of $v\in \GGG$ and calling the resulting orbit segment the $\hat{\GGG}$-part of $((x,a),T)$, then observing that the parts of $((x,a),T)$ that lie before and after this piece are in $\hat\PPP$ and $\hat\SSS$, respectively. More precisely, there are two cases to consider.
\begin{itemize}
\item If $|v| \geq 4$, then let $P = |u^p| + 2 - a$, $G = |v| - 4$, and $S = |u^s| + 2 + a + T - n$.
\item If $|v| < 4$, then let $k = \lfloor |v|/2 \rfloor$ and put $P = |u^p| + k - a$; similarly, put $\ell = \lceil |v|/2 \rceil$ and $S = |u^s| + \ell + a + T - n$. In this case $G=0$.
\end{itemize}
In both cases one obtains $((x,a), P) \in \hat\PPP$, $(f_P(x,a),G) \in \hat\GGG$, and $(f_{P+G}(x,a),S) \in \hat\SSS$.

\subsubsection{Specification}

We prove that $\hat\GGG$ has specification at scale $\delta$. By Condition \ref{spec} of Theorem \ref{thm:CT-PYY}, there is $t \in \NN$ such that for any collection of words $w^1,\dots, w^k \in \GGG$, there are words $u^1,\dots, u^{k-1}\in \LLL_t$ such that $w^1 u^1 w^2 \dots w^{k-1}u^{k-1}w^k \in \LLL$. 

Given any collection of orbit segments $((x_i,0),n_i) \in \hat{\GGG}$ for  $1\leq i \leq k$,
let $w^i = (x_i)_{[-2,n_i + 2]}$, so that $w^i \in \GGG$. Let $u^1,\dots, u^{k-1} \in \LLL_t$ be the words provided by Condition \ref{spec}, so that $w^1 u^1 w^2 \cdots w^{k-1}u^{k-1}w^k \in \LLL$; then fix $y\in [w^1 u^1 w^2 \cdots u^{k-1} w^k]$, and observe that $(\sigma^2 y, 0)$ has the desired shadowing property by Lemma \ref{lem:hat-d-Bn}, with gap size $t+4$.

\subsubsection{Bowen property}\label{sec:Bowen-Prop}

Fix $((x,0),n)\in \hat\GGG$ and $(y,a) \in \bar{B}_n((x,0),\eps)$, where $x,y\in X$ and $|a| \leq \frac 12$. For each $k = 0,1,2,\dots, n-1$, we have $\hat{d}((\sigma^k x, \frac 12), (\sigma^k y, a+ \frac 12)) \leq \eps$, and so Lemma \ref{lem:hat-d-leq} gives $x_0 x_1 \cdots x_n = y_0 y_1 \cdots y_n$.

Observe that $|\Phi_0((y,a),n) - \Phi_0((y,0),n)| \leq 2a\|\ph\| \leq \|\ph\|$. By  Lemma \ref{lem:hat-ph},
\[
|\Phi_0((x,0),n) - \Phi_0((y,0),n)|
\leq |S_n\ph(x) - S_n\ph(y)| + 2\|\ph\|.
\]
We conclude that
\begin{equation}\label{eqn:bowen-bowen}
|\Phi_0((y,a),n) - \Phi_0((x,0),n)| \leq |S_n\ph(x) - S_n\ph(y)| + 3\|\ph\|.
\end{equation}
Since $((x,0),n) \in \hat\GGG$, we have $x_{-2} x_{-1} \cdots x_n x_{n+1} x_{n+2} \in \GGG$. Using \eqref{eqn:bowen-bowen}, we get
\[
|\Phi_0((y,a),n) - \Phi_0((x,0),n)|
\leq |S_{n+4}\ph(\sigma^{-2} x) - S_{n+4}\ph(\sigma^{-2} y)|
+ 11\|\ph\|
\]
and now Condition \ref{Bow} (the Bowen property on $\GGG$) gives
\[
|\Phi_0((y,a),n) - \Phi_0((x,0),n)| \leq C + 11\|\ph\|,
\]
which proves that $\hat\ph$ has the Bowen property on $\hat\GGG$ at scale $\eps$.

\subsubsection{Pressure gap}
We prove the pressure gap condition by way of the following.

\begin{lemma}\label{lem:P-leq}
$P([\hat\SSS],\hat\ph,\delta,\eps) \leq P(\CCC^s,\ph)$ and
$P([\hat\PPP],\hat\ph,\delta,\eps) \leq P(\CCC^p,\ph)$.
\end{lemma}

Given Lemma \ref{lem:P-leq}, we have
$P([\hat\PPP] \cup [\hat\SSS],\hat\ph,\delta\eps)
\leq P(\CCC^p \cup \CCC^s,\ph)
< P(\ph) = P(\hat\ph)$,
where the second inequality is condition \ref{gap} of Theorem \ref{thm:CT-PYY}.
The rest of this section is devoted to the proof of Lemma \ref{lem:P-leq} for $[\hat\SSS]$; the proof for $[\hat\PPP]$ is similar.

Given $((x,a),n) \in [\hat\SSS]$ with $a\in [0,1]$, we have $u := x_j \cdots x_{k-1} \in \CCC^s$ for some $j\in \{0,1,2\}$ and $k\in \{n,n+1\}$. If $(y,b) \in B_n((x,a),\eps)$ and $|b-a|\leq \frac 12$, then $(\sigma y,b-a) \in B_{n-1}((\sigma x,0),\eps)$, and arguing as in \S\ref{sec:Bowen-Prop} gives $y_{[1,n)} = x_{[1,n)}$. 
Moreover, $b\in [-\frac 12, \frac 32]$, so using Lemma \ref{lem:hat-ph} we get
\begin{equation}\label{eqn:int-0k}
\Phi_0((y,b),n) \leq 3\|\ph\| + \Phi_0((y,0),n) \leq 4\|\ph\| + S_n\ph(y).
\end{equation}
Since $S_n\ph(y) + S_{k-n}\ph(\sigma^n y) = S_j \ph(y) + S_{k-j}\ph(\sigma^j y)$, we have
\[
|S_n\ph(y) - S_{k-j} \ph(\sigma^j y)| \leq (j + (k-n))\|\ph\| \leq 3\|\ph\|,
\]
and with $V_n = V_n(\ph)$ as in \S\ref{sec:classes}, we can write $A_N := \sum_{i=0}^{N-1} V_i = o(N)$ and obtain
\[
|S_{k-j} \ph(\sigma^j y) - S_{k-j} \ph(\sigma^j x)|
\leq \sum_{i=0}^{k-j-1} V_{\min(1+i,k-j-i)} \leq 2 A_{k-j}.
\]
Combining these gives
\[
S_n\ph(y) \leq S_{k-j} \ph(\sigma^j x) + 3\|\ph\| + 2 A_{k-j}
\leq \Phi(u) + 3\|\ph\| + 2A_{n+1},
\]
and together with \eqref{eqn:int-0k} we obtain
\begin{equation}\label{eqn:Phi-eps}
\Phi_\eps((x,a),n) \leq 7\|\ph\| + \Phi(u) + 2A_{n+1}.
\end{equation}
This will let us bound the partition sum associated to an $(n,\delta)$-separated set $E\subset \hat{X}$ once we control the multiplicity of the map $((x,a),n) \mapsto u$.

\begin{lemma}\label{lem:delta-sep}
There exists $D>0$ such that given any $n\in \NN$ and any $(n,\delta)$-separated set $E\subset \hat{X}$, we have 
$\#\{ (x,a) \in E : a\in [0,1], x_{[j,k)} = u \} \leq D$
for each $j\in \{0,1,2\}$, and $k\in \{n,n+1\}$.
\end{lemma}
\begin{proof}
Fix $\xi \in (\theta^3, \delta)$.
If $(x,a),(y,b) \in E$ are distinct and have $x_{[-2, n+3]} = y_{[-2, n+3]}$, then Lemma \ref{lem:hat-d-Bn} gives $(y,0) \in B_{n+1}((x,0),\xi)$, and thus $(y,b) \in B_n((x,a),\xi + |b-a|)$;  this implies $\xi + |b-a| \geq \delta$, so $|b-a| \geq \delta - \xi$. It follows that $E$ can contain at most $1/(\delta-\xi)$ points for each choice of the symbols in positions $-2,-1,\dots, n+3$. Given $j,k$, if $x_{[j,k)} = u = y_{[j,k)}$ then the only symbols where we could possibly have $x_i \neq y_i$ occur when $i \in \{-2,-1,0,1,n+1,n+2,n+3\}$. There are at most $(\#A)^7$ ways of filling these symbols, which proves the lemma with $D = (\#A)^7/(\delta-\xi)$.
\end{proof}

Using Lemma \ref{lem:delta-sep} and \eqref{eqn:Phi-eps} gives
\[
\Lambda([\hat\SSS],\hat\ph,\delta,\eps,n)
\leq \sum_{j=0}^2 \sum_{k=n}^{n+1} D e^{7\|\ph\| + 2A_{n+1}} \Lambda_{k-j}(\CCC^s),
\]
and since $A_{n+1} = o(n)$ this proves Lemma \ref{lem:P-leq}.

\bibliographystyle{amsalpha}
\bibliography{es-formula-ref}

\end{document}